\documentclass[11pt]{amsart}

\usepackage{amsmath,amssymb,amsthm,amscd,bm,mathtools,graphicx, xcolor}
\usepackage{pb-diagram}
\usepackage[left=2.54cm, right=2.54cm, top=2.54cm, bottom=2.54cm]{geometry}

 \usepackage{hyperref}
 \hypersetup{colorlinks=true, citecolor=darkblue, linkcolor=darkblue}
 \definecolor{darkblue}{rgb}{0.0,0,0.7} 

\usepackage{color}

\allowdisplaybreaks
\numberwithin{equation}{section}

\newcommand{\Fg}{\mathfrak{g}}
\newcommand{\Fh}{\mathfrak{h}}

\newcommand{\BZ}{\mathbb{Z}}

\newcommand{\BR}{\mathbb{R}}
\newcommand{\BC}{\mathbb{C}}
\newcommand{\BB}{\mathbb{B}}

\newcommand{\CA}{\mathcal{A}}
\newcommand{\CL}{\mathcal{L}}
\newcommand{\CB}{\mathcal{B}}

\newcommand{\Le}{\mathsf{L}}
\newcommand{\OS}{\mathsf{OS}}

\newcommand{\bzero}{\mathbf{0}}
\newcommand{\bc}{\mathbf{c}_{0}}

\newcommand{\vpi}{\varpi}

\newcommand{\GL}{\mathrm{GL}}
\newcommand{\cl}{\mathrm{cl}}
\newcommand{\af}{\mathrm{af}}
\newcommand{\ext}{\mathrm{ext}}
\newcommand{\dir}{\mathop{\rm dir}\nolimits}
\newcommand{\degr}{\mathop{\rm deg}\nolimits}
\newcommand{\wt}{\mathop{\rm wt}\nolimits}
\newcommand{\qwt}{\mathop{\rm qwt}\nolimits}
\newcommand{\Ht}{\mathop{\rm height}\nolimits}
\newcommand{\cHt}{\mathop{\rm coheight}\nolimits}
\newcommand{\QB}{\mathrm{QB}}
\newcommand{\QLS}{\mathrm{QLS}}
\newcommand{\Deg}{\mathop{\rm Deg}\nolimits}
\newcommand{\norm}{\mathrm{norm}}
\newcommand{\Par}{\mathrm{Par}}
\newcommand{\lex}{\mathrm{lex}}

\newcommand{\QBM}{\mathrm{QB}(e\,;\,m_{\lamm})}
\newcommand{\QBX}{\mathrm{QB}(e\,;\,m_{\mu})}
\newcommand{\QBw}{\mathrm{QB}(e\,;\,m_{w\lambda})}

\newcommand{\gch}{\mathop{\rm gch}\nolimits}

\newcommand{\lng}{w_{\circ}}
\newcommand{\lngJ}{w_{J,\circ}}
\newcommand{\lamm}{\lng\lambda}
\newcommand{\p}[1]{\pi(i_{#1})}

\newcommand{\mcr}[1]{\lfloor #1 \rfloor}
\newcommand{\pair}[2]{\langle #1,\,#2 \rangle}
\newcommand{\bpair}[2]{\Bigl\langle #1,\,#2 \Bigr\rangle}

\newcommand{\edge}[1]{\xrightarrow{\ #1 \ }}
\newcommand{\edger}[1]{\xleftarrow{\ #1 \ }}
\newcommand{\blarrl}[1]{\xLeftarrow{#1}}
\newcommand{\enp}[1]{\mathrm{end}(p_{#1})}
\newcommand{\Mac}[1]{E_{#1}(x\,;\,q,\,0)}
\newcommand{\Mact}[1]{E_{#1}(x\,;\,q,\,t)}

\newcommand{\rr}{\Phi_{\af}}
\newcommand{\prr}{\Phi_{\af}^{+}}

\newcommand{\si}{\frac{\infty}{2}}
\newcommand{\sell}{\ell^{\si}}
\newcommand{\sil}{\prec}
\newcommand{\sile}{\preceq}
\newcommand{\sig}{\succ}

\newcommand{\SB}{\mathrm{SiB}^{J}}
\newcommand{\SBb}[1]{\mathrm{SiB}(\lambda\,;\,#1)}

\newcommand{\sLS}{\mathbb{B}^{\si}(\lambda)}

\newcommand{\ol}[1]{\overline{#1}}

\newcommand{\ti}[1]{\widetilde{#1}}
\newcommand{\ub}[2]{\underbrace{#1}_{#2}}

\newenvironment{enu}{%
 \begin{enumerate}%
}{\end{enumerate}}

\theoremstyle{plain}
\newtheorem{lem}{Lemma}[section]

\newtheorem{prop}[lem]{Proposition}
\newtheorem{thm}[lem]{Theorem}

\theoremstyle{definition}
\newtheorem{dfn}[lem]{Definition}

\theoremstyle{remark}
\newtheorem{rem}[lem]{Remark}

\begin{document}

\title[A uniform model for KR crystals III: Nonsymmetric Macdonald polynomials]
{A uniform model for Kirillov--Reshetikhin crystals III: \\ 
 Nonsymmetric Macdonald polynomials at $t=0$ and \\ 
 Demazure characters}

\author[C.~Lenart]{Cristian Lenart}
\address[Cristian Lenart]
 {Department of Mathematics and Statistics, 
  State University of New York at Albany, 
  Albany, NY 12222, U.S.A.}
\email{clenart@albany.edu}
\urladdr{http://www.albany.edu/\~{}lenart/}

\author[S.~Naito]{Satoshi Naito}
\address[Satoshi Naito]
 {Department of Mathematics, Tokyo Institute of Technology,
  2-12-1 Oh-Okayama, Meguro-ku, Tokyo 152-8551, Japan}
\email{naito@math.titech.ac.jp}

\author[D.~Sagaki]{Daisuke Sagaki}
\address[Daisuke Sagaki]
 {Institute of Mathematics, University of Tsukuba, 
  Tsukuba, Ibaraki 305-8571, Japan}
\email{sagaki@math.tsukuba.ac.jp}

\author[A.~Schilling]{Anne Schilling}
\address[Anne Schilling]
 {Department of Mathematics, University of California, One Shields Avenue, 
  Davis, CA 95616-8633, U.S.A.}
\email{anne@math.ucdavis.edu}
\urladdr{http://www.math.ucdavis.edu/\~{}anne}

\author[M.~Shimozono]{Mark Shimozono}
\address[Mark Shimozono]
 {Department of Mathematics, MC 0151, 460 McBryde Hall, Virginia Tech,
  225 Stanger St., Blacksburg, VA 24061 USA}
\email{mshimo@vt.edu}

%
\begin{abstract} 
We establish the equality of the specialization $\Mac{w\lambda}$ of 
the nonsymmetric Macdonald polynomial $\Mact{w\lambda}$ at $t=0$ with the graded character 
$\gch U_{w}^{+}(\lambda)$ of a certain Demazure-type submodule 
$U_{w}^{+}(\lambda)$ of a tensor product of ``single-column'' Kirillov--Reshetikhin 
modules for an untwisted affine Lie algebra, where $\lambda$ is a dominant integral weight and 
$w$ is a (finite) Weyl group element; this generalizes our previous result,
that is, the equality between the specialization $P_{\lambda}(x\,;\,q,\,0)$ of the 
symmetric Macdonald polynomial $P_{\lambda}(x\,;\,q,\,t)$ 
at $t=0$ and the graded character of a tensor product 
of single-column Kirillov--Reshetikhin modules. We also give two combinatorial formulas 
for the mentioned specialization of nonsymmetric Macdonald polynomials: 
one in terms of quantum Lakshmibai--Seshadri paths and the other in terms of the quantum alcove model. 
\end{abstract}%

\maketitle
%
\section{Introduction.}
\label{sec:intro}

In our previous paper \cite{LNSSS2}, we proved that 
the specialization $P_{\lambda}(x\,;\,q,\,0)$ of the 
symmetric Macdonald polynomial $P_{\lambda}(x\,;\,q,\,t)$ at $t=0$ 
is identical to the graded character of a certain tensor product 
of Kirillov--Reshetikhin (KR for short) modules of one-column type 
for an untwisted affine Lie algebra $\Fg_{\af}$, 
where $\lambda$ is a dominant integral weight 
for the finite-dimensional simple Lie algebra $\Fg \subset \Fg_{\af}$.
The purpose of this paper is to generalize this result to 
the specialization $\Mac{w\lambda}$ of 
the nonsymmetric Macdonald polynomial 
$\Mact{w\lambda}$ at $t=0$, where $w$ is an element 
of the (finite) Weyl group $W$ of $\Fg$; 
note that if $w$ is the longest element $\lng$ of $W$, then 
$\Mac{\lng\lambda}=P_{\lambda}(x\,;\,q,\,0)$. 

Let us explain our result more precisely. Let $\Fg$ be
 a finite-dimensional simple Lie algebra (over $\BC$), 
 with $X$ its integral weight lattice, 
and $\Fg_{\af}$ the associated untwisted affine Lie algebra. 
We denote by $\bigl\{\alpha_{i}\bigr\}_{i \in I}$ and 
$\bigl\{\alpha_{i}^{\vee}\bigr\}_{i \in I}$ 
the simple roots and simple coroots of $\Fg$, respectively, 
and by $\vpi_{i}$, $i \in I$, the fundamental weights for $\Fg$. 
For a dominant integral weight $\lambda=\sum_{i \in I} m_{i} \vpi_{i} \in X$
with $m_{i} \in \BZ_{\ge 0}$, let $\QLS(\lambda)$ denote 
the crystal of quantum Lakshmibai-Seshadri (QLS for short) paths 
of shape $\lambda$; for details, see Definition~\ref{dfn:QLS} below. 
Then we know from \cite{LNSSS2} that 
the crystal $\QLS(\lambda)$ provides a realization of the crystal basis of 
the tensor product $\bigotimes_{i \in I} W(\vpi_{i})^{\otimes m_{i}}$ 
of the level-zero fundamental representations $W(\vpi_{i})$, $i \in I$, 
of the quantum affine algebra $U_{q}'(\Fg_{\af})$ associated to $\Fg_{\af}$. 
The main result of \cite{LNSSS2} states that the specialization 
$P_{\lambda}(x\,;\,q,\,0)$ of the symmetric Macdonald polynomial 
at $t=0$ is identical to the graded character of the crystal $\QLS(\lambda)$, 
where the grading on $\QLS(\lambda)$ is given by the degree function, 
or equivalently, by the (global) energy function. 

Let $W=\langle r_{i} \mid i \in I \rangle$ denote the (finite) Weyl group of $\Fg$, 
and set $W_{J}:=\langle r_{i} \mid i \in J \rangle \subset W$, where 
$J:=\bigl\{i \in I \mid \pair{\alpha_{i}^{\vee}}{\lambda}=0 \bigr\}$. 
Also, let $W^{J}$ denote the set of minimal(-length) coset representatives 
for the cosets in $W/W_{J}$; for $w \in W$, 
we denote by $\mcr{w}=\mcr{w}^{J} \in W^J$ the minimal coset representative 
for the coset $w W_J$ in $W/W_J$. Now, for $w \in W^{J}$, we set
\begin{equation*}
\QLS_{w}(\lambda) := 
 \bigl\{ \eta \in \QLS(\lambda) \mid 
   \iota(\eta) \le w 
 \bigr\}, 
\end{equation*}
where for a QLS path 
$\eta = (x_{1},\,\dots,\,x_{s} \,;\, \sigma_{0},\,\sigma_{1},\,\dots,\,\sigma_{s}) 
\in \QLS(\lambda)$, we define the initial direction $\iota(\eta)$ of $\eta$ 
to be $x_{1} \in W^{J}$; here the symbol $\le$ is used to 
denote the Bruhat order on $W$. Furthermore, we define the graded character 
$\gch \QLS_{w}(\lambda)$ of $\QLS_{w}(\lambda) \subset \QLS(\lambda)$ by
\begin{equation*}
\gch \QLS_{w}(\lambda) := \sum_{\eta \in \QLS_{w}(\lambda)} 
  q^{-\Deg(\eta)} e^{\wt (\eta)}, 
\end{equation*}
where $\wt:\QLS(\lambda) \rightarrow X$ and 
$\Deg:\QLS(\lambda) \rightarrow \BZ_{\le 0}$ denote 
the weight function and the degree function on $\QLS(\lambda)$, respectively; 
for the definitions, see \eqref{eq:wt} and \eqref{eq:Deg} below. 
Now, the main result of this paper is as follows.
%
%
\begin{thm} \label{thm:Mac0}
For each $w \in W^{J}$, the equality
\begin{equation*}
\gch \QLS_{w}(\lambda) = \Mac{w\lambda}
\end{equation*}
holds, where $\Mac{w\lambda}$ denotes the specialization of 
the nonsymmetric Macdonald polynomial $\Mact{w\lambda}$ at $t=0$. 
\end{thm}

We should mention that this result generalizes 
\cite[Proposition~7.9]{LNSSS2}, 
since it holds that $\QLS_{\mcr{\lng}}(\lambda) = \QLS(\lambda)$ and 
$\Mac{\mcr{\lng}\lambda}=P_{\lambda}(x\,;\,q,\,0)$, 
where $\lng \in W$ denotes the longest element. On the other hand, 
in Theorems~\ref{thmalc} and \ref{thmalcgen}, we express 
$\Mac{w\lambda}$ in terms of the so-called quantum alcove model \cite{LL}.

In the following, we explain the representation-theoretic meaning of 
Theorem~\ref{thm:Mac0}; see \S\ref{sec:gc} for details. 
Let $V(\lambda)$ denote the extremal weight module of extremal weight $\lambda$ 
over the quantum affine algebra $U_{q}(\Fg_{\af})$ associated to $\Fg_{\af}$, and set 
$V_{w}^{+}(\lambda):=U_{q}^{+}(\Fg_{\af})S_{w}^{\norm}v_{\lambda} \subset V(\lambda)$ 
for $w \in W$, which is the Demazure submodule generated by the extremal weight vector 
$S_{w}^{\norm}v_{\lambda} \in V(\lambda)$ of weight $w\lambda$ 
over the positive part $U_{q}^{+}(\Fg_{\af})$ of $U_{q}(\Fg_{\af})$; 
note that $V_{w}^{+}(\lambda) \subset V_{\lng}^{+}(\lambda)$ for all $w \in W$. 
For $w \in W$, we define $U_{w}^{+}(\lambda)$ 
to be the image of $V_{w}^{+}(\lambda)$ under the canonical projection 
$V_{\lng}^{+}(\lambda) \twoheadrightarrow
V_{\lng}^{+}(\lambda)/Z_{\lng}^{+}(\lambda)$; 
for the definition of $Z_{\lng}^{+}(\lambda)$, see \S\ref{subsec:gch}. 
Then, $U_{\lng}^{+}(\lambda)$ is isomorphic, as a $U_{q}(\Fg)$-module, to
the tensor product $\bigotimes_{i \in I} W(\vpi_{i})^{\otimes m_{i}}$ 
of level-zero fundamental representations $W(\vpi_{i})$, $i \in I$; 
note that this is not an isomorphism of $U_{q}^{+}(\Fg_{\af})$-modules.
Because the module $V_{w}^{+}(\lambda)$ is generated 
by the extremal weight vector $S_{w}^{\norm}v_{\lambda} \in V(\lambda)$ 
over $U_{q}^{+}(\Fg_{\af})$, it follows that the module
$U_{w}^{+}(\lambda) \subset U_{\lng}^{+}(\lambda)$ is also generated by 
the image of $S_{w}^{\norm}v_{\lambda}$ over $U_{q}^{+}(\Fg_{\af})$. 
Thus, in a sense, 
we can think of $U_{w}^{+}(\lambda) \subset U_{\lng}^{+}(\lambda)$
as a Demazure-type submodule of $U_{\lng}^{+}(\lambda)$, 
which is isomorphic as a $U_{q}(\Fg)$-module to 
$\bigotimes_{i \in I} W(\vpi_{i})^{\otimes m_{i}}$. 
Also, if we define the graded character 
$\gch U_{w}^{+}(\lambda)$ of $U_{w}^{+}(\lambda)$ by
\begin{equation*}
\gch U_{w}^{+}(\lambda) := 
 \sum_{\gamma \in Q,\,k \in \BZ} 
 \dim U_{w}^{+}(\lambda)_{\lambda-\gamma+k\delta}\,x^{\lambda-\gamma}q^{k},
\end{equation*}
where $Q:=\bigoplus_{i \in I}\BZ \alpha_{i}$ is the root lattice for $\Fg$, 
$\delta$ denotes the null root of $\Fg_{\af}$, and $q:=x^{\delta}$, then we have 
(see Theorem~\ref{thm:gch})
\begin{equation*}
\gch U_{w}^{+}(\lambda) = 
\gch \QLS_{w}(\lambda) \stackrel{\text{Theorem~\ref{thm:Mac0}}}{=} \Mac{w\lambda}.
\end{equation*}

In \S\ref{sec:Mac0}, we give a bijective proof of Theorem~\ref{thm:Mac0}
by making use of the Orr-Shimozono formula for the specialization at $t=0$ of 
nonsymmetric Macdonald polynomials~\cite{OS}.
The outline of our proof is as follows. 
In \S\ref{subsec:OS}, we briefly review 
the Orr-Shimozono formula (see Theorem~\ref{thm:OS}), 
which expresses the specialization $\Mac{\mu}$ of 
the nonsymmetric Macdonald polynomial $\Mact{\mu}$ at $t=0$ 
in terms of the set $\QBX$ of quantum alcove paths
from $e$ to $m_{\mu}$ for an integral weight $\mu$,
where $m_{\mu}$ denotes the element of the (extended) affine
Weyl group that is of minimal length in the coset $t_{\mu} W$,
with $t_{\mu}$ the translation by $\mu$.
Next, for a dominant integral weight $\lambda \in X$, 
we show in Lemma~\ref{lem:A-QB} that there exists a canonical bijection 
between the particular set $\QBM_{\lex}$ and the set $\CA(-\lamm)$; 
here, $\QBM_{\lex}$ is defined by using a specific reduced expression 
for $m_{\lamm} = t_{\lamm}$ corresponding to 
a lexicographic $(-\lamm)$-chain of roots. 
%
Also, we give an explicit bijection 
$\Xi:\QBM_{\lex} \rightarrow \QLS(\lambda)$ 
in such a way that the diagram below 
is commutative (see Proposition~\ref{prop:Xi}). 
Furthermore, in Lemma~\ref{lem:embed} 
combined with Proposition~\ref{prop:Theta}, 
we show that there exists a natural embedding 
$\QBw \hookrightarrow \QBM_{\lex}$ for an arbitrary $w \in W^{J}$. 
\begin{equation*}
\begin{diagram}
\node{\QBw \quad} 
\arrow{e,b}{\begin{subarray}{c}
  \text{Embedding} \\
  \text{(Lemma~\ref{lem:embed}} \\
  \text{and Proposition~\ref{prop:Theta})}
  \end{subarray}
  }
\node{\quad \QBM_{\lex} \quad} 
\arrow{e,b}{\begin{subarray}{c}
  \text{Bijection} \\
  \text{(Lemma~\ref{lem:A-QB})}
  \end{subarray}
  }
\arrow{se,b}{\begin{subarray}{c}
  \text{Bijection $\Xi$} \\
  \text{(Proposition~\ref{prop:Xi})}
  \end{subarray}
  }
\node{\CA(-\lamm)} 
\arrow{s,r}{\begin{subarray}{c}
  \text{Bijection $\Pi$} \\
  \text{(\cite[\S8.1]{LNSSS2})}
  \end{subarray}
  } \\
\node{} \node{}
\node{\QLS(\lambda)}
\end{diagram}
\end{equation*}
Finally, in Proposition~\ref{prop:embed} and Lemma~\ref{lem:Xi}, 
we show that the image of $\QBw$ under the composite of the maps 
$\QBw \hookrightarrow \QBM_{\lex} \stackrel{\Xi}{\longrightarrow} \QLS(\lambda)$
is identical to $\QLS_{w}(\lambda)$; we also show in 
Proposition~\ref{prop:Theta}, Lemma~\ref{lem:embed}, and Proposition~\ref{prop:Xi} that 
both of the embedding $\QBw \hookrightarrow \QBM_{\lex}$ and 
the bijection $\Xi:\QBM \rightarrow \QLS(\lambda)$ 
preserve ``weights'' and ``degrees''. 
This implies that the graded character of $\QLS_{w}(\lambda)$ is 
identical to that of $\QBw$. Because we know from the Orr-Shimozono formula
that the graded character of $\QBw$ is identical to $\Mac{w\lambda}$, 
we conclude from the above that the graded character of 
$\QLS_{w}(\lambda)$ is identical to $\Mac{w\lambda}$. 

In Appendix~\ref{sec:rec-gch}, using the crystal structure on the set $\QLS(\lambda)$, 
we obtain a recursive formula (see Proposition~\ref{prop:dem}) 
for the graded characters $\gch \QLS_{w}(\lambda)$, $w \in W^{J}$, 
which is described in terms of Demazure operators. 
Here we note that in view of Theorem~\ref{thm:Mac0} above, this recursive formula
is equivalent to the one (see Proposition~\ref{prop:demM}) 
for nonsymmetric Macdonald polynomials $\Mac{w\lambda}$, $w \in W^{J}$, 
specialized at $t=0$; in Appendix~\ref{subsec:demM-Mac0}, 
we provide a sketch of how to derive this recursive formula for 
$\Mac{w\lambda}$ by using the polynomial representation of 
the double affine Hecke algebra.

\subsection*{Acknowledgments}
C.L. was partially supported by the NSF grant DMS--1362627. 
S.N. was partially supported by Grant-in-Aid for Scientific Research (C), No. 24540010, Japan.
D.S. was partially supported by Grant-in-Aid for Young Scientists (B), No. 23740003, Japan. 
A.S. was partially supported by NSF grants OCI--1147247 and DMS--1500050.
M.S. was partially supported by the NSF grant DMS--1200804.
%
%
\section{Proof of Theorem~\ref{thm:Mac0}.}
\label{sec:Mac0}
%
%
\subsection{Setting.}
\label{subsec:setting}

Let $\Fg$ be a finite-dimensional simple Lie algebra (over $\BC$). 
We denote by $\bigl\{\alpha_{i}\bigr\}_{i \in I}$ and 
$\bigl\{\alpha_{i}^{\vee}\bigr\}_{i \in I}$ 
the simple roots and simple coroots of $\Fg$, respectively, 
and by $\vpi_{i}$, $i \in I$, the fundamental weights for $\Fg$; 
we set
\begin{equation*}
Q:=\bigoplus_{i \in I} \BZ\alpha_{i}, \qquad
Q^{\vee}:=\bigoplus_{i \in I} \BZ \alpha_{i}^{\vee}, 
\quad \text{and} \quad
X:=\bigoplus_{i \in I} \BZ \vpi_{i}.
\end{equation*}
Let $\Phi^{+}$ (resp., $\Phi^{\vee+}$) denote
the set of positive roots (resp., coroots), and 
$\Phi^{-}$ (resp., $\Phi^{\vee-}$) 
the set of negative roots (resp., coroots). 
We set $\rho:=(1/2) \sum_{\alpha \in \Phi^{+}}\alpha$.
Let $W=\langle r_{i} \mid i \in I \rangle$ be 
the (finite) Weyl group of $\Fg$, 
with length function $\ell:W \rightarrow \BZ_{\ge 0}$;
we denote by $\lng \in W$ the longest element, 
and by $e \in W$ the identity element. Also, let us 
denote by $\omega:I \rightarrow I$ the Dynkin diagram automorphism 
given by: $w_{\circ}\alpha_{i} = -\alpha_{\omega(i)}$ for $i \in I$. 

For a subset $J \subset I$, we set 
\begin{align*}
& \Phi_{J}^{+}:=\Phi^{+} \cap \biggl(\bigoplus_{i \in J} \BZ\alpha_{i}\biggr), & & 
  \rho_{J}:=\frac{1}{2}\sum_{\alpha \in \Phi_{J}^{+}} \alpha, \\[1mm]
& \Phi_{J}^{\vee+}:=\Phi^{\vee+} \cap \biggl(\bigoplus_{i \in J} \BZ\alpha_{i}^{\vee}\biggr), & & 
  W_{J}:=\langle r_{i} \mid i \in J \rangle \subset W;
\end{align*}
let $\lngJ$ denote the longest element of $W_{J}$. 
Also, let $W^J$ denote the set of minimal(-length) coset 
representatives for the cosets in $W / W_J$; recall that
%
%
\begin{equation} \label{eq:mcrs}
W^{J}=\bigl\{ w \in W \mid 
  \text{$w\alpha \in \Phi^{+}$ for all $\alpha \in \Phi_{J}^{+}$} \bigr\}, 
\end{equation}
\begin{equation}
\ell(wz) = \ell(w) + \ell(z) \qquad 
 \text{for all $w \in W^{J}$ and $z \in W_{J}$}.
\end{equation}
For $w \in W$, we denote by $\mcr{w}=\mcr{w}^{J} \in W^J$ 
the minimal coset representative for the coset $w W_J$ in $W/W_J$.
We use the symbol $\le$ for the Bruhat order on the Weyl group $W$.
%
%
\subsection{Quantum Lakshmibai-Seshadri paths.}
\label{sec:QLS}
In this subsection, we recall the definition of quantum Lakshmibai-Seshadri paths 
from \cite[\S3]{LNSSS2}. 
%
%
\begin{dfn}\label{def:QB}
Let $J$ be a subset of $I$.
The (parabolic) quantum Bruhat graph $\QB(W^{J})$ is 
the $(\Phi^{+} \setminus \Phi_{J}^{+})$-labeled, 
directed graph with vertex set $W^J$ and 
$(\Phi^{+} \setminus \Phi_{J}^{+})$-labeled, directed edges 
of the following form: 
$w \edge{\beta} \mcr{wr_{\beta}}$ 
for $w \in W^{J}$ and $\beta \in \Phi^{+} \setminus \Phi_{J}^{+}$, 
where either

(i) $\ell(\mcr{wr_{\beta}})=\ell(w)+1$, or 

(ii) $\ell(\mcr{wr_{\beta}})=\ell(w)-2\pair{\beta^{\vee}}{\rho-\rho_{J}}+1$;

\noindent
if (i) holds (resp., (ii) holds), then the edge is called a Bruhat edge 
(resp., a quantum edge). If $J$ is the empty set $\emptyset$, then 
we simply write $\QB(W^{J}) = \QB(W^{\emptyset})$ as $\QB(W)$. 
\end{dfn}
%
%
\begin{rem} \label{rem:QB}
(1) We have $\pair{\beta^{\vee}}{\rho-\rho_{J}} > 0$ 
for all $\beta \in \Phi^{+} \setminus \Phi_{J}^{+}$. 
Indeed, since $\pair{\alpha_{i}^{\vee}}{\alpha} \le 0$ 
for all $i \in I \setminus J$ and $\alpha \in \Phi_{J}^{+}$, 
we see that $\pair{\alpha_{i}^{\vee}}{\rho_{J}} \le 0$ for all 
$i \in I \setminus J$, and hence 
$\pair{\alpha_{i}^{\vee}}{\rho-\rho_{J}} > 0$ for all $i \in I \setminus J$. 
Also, we have $\pair{\alpha_{i}^{\vee}}{\rho-\rho_{J}} = 1 -1 = 0$ for all $i \in J$. 
Therefore, $\pair{\beta^{\vee}}{\rho-\rho_{J}} > 0$ 
for all $\beta \in \Phi^{+} \setminus \Phi_{J}^{+}$. 
As a consequence, if $w \edge{\beta} \mcr{wr_{\beta}}$ is a quantum edge, 
then $\ell(\mcr{wr_{\beta}}) < \ell(w)$. 

(2) If $w \edge{\beta} \mcr{wr_{\beta}}$ is a Bruhat edge, 
then $wr_{\beta} \in W^{J}$, and hence $\mcr{wr_{\beta}} = wr_{\beta}$ 
(see \cite[Remark~3.1.2]{LNSSS3}). 

(3) Let $x,\,y \in W^{J}$ be such that $x \le y$ in the Bruhat order on $W$. If 
%
%
\begin{equation} \label{eq:dpA}
x=x_{0} \edge{\beta_{1}} x_{1} \edge{\beta_{2}} \cdots 
\edge{\beta_{k}} x_{k}=y
\end{equation}
is a shortest directed path from $x$ to $y$ 
in $\QB(W^{J})$, then all of its edges are Bruhat edges. Indeed, 
by Definition~\ref{def:QB} (for Bruhat edges) and part (1) of this remark 
(for quantum edges), we have
\begin{equation} \label{eq:length}
\ell(y) - \ell(x) = \sum_{q=1}^{k} 
 (\ub{\ell(x_{q})-\ell(x_{q-1})}{=1 \text{ or } < 0}) 
 \le \sum_{q=1}^{k} 1 = k; 
\end{equation}
note that the equality holds if and only if 
$\ell(x_{q})-\ell(x_{q-1})=1$ for all $1 \le q \le k$, 
or equivalently, all the edges are Bruhat edges. 
Since $x \le y$ by the assumption, 
we deduce from the chain property (see \cite[Theorem~2.5.5]{BB}) 
that there exists a directed path from $x$ to $y$ in $\QB(W^{J})$ 
all of whose edges are Bruhat edges; the length of this directed path 
is equal to $\ell(y)-\ell(x)$. Therefore, we obtain $k \le \ell(y)-\ell(x)$ 
since the directed path \eqref{eq:dpA} is a shortest one. 
Combining this inequality and \eqref{eq:length}, we obtain 
$k=\ell(y)-\ell(x)$, and hence all the edges 
in the shortest directed path \eqref{eq:dpA} are Bruhat edges.
\end{rem}

Now, we fix a dominant integral weight $\lambda \in X$ for $\Fg$, and set 
\begin{equation*}
J=J_{\lambda}:=\bigl\{i \in I \mid \pair{\alpha_{i}^{\vee}}{\lambda}=0\bigr\} \subset I.
\end{equation*}
As above, we simply write $\mcr{w}^{J}=\mcr{w}^{J_{\lambda}} \in W^{J}$ 
for $w \in W$ as $\mcr{w}$, unless stated otherwise explicitly.
%
%
\begin{dfn} \label{dfn:achain}
For a given rational number $\sigma$, we define $\QB_{\sigma\lambda}(W^{J})$ to be the subgraph
of the parabolic quantum Bruhat graph $\QB(W^{J})$ with the same vertex set but 
having only the edges: 
\begin{equation*}
w \edge{\beta} \mcr{wr_{\beta}} \quad \text{with} \quad 
\pair{\beta^{\vee}}{\sigma\lambda}=\sigma \pair{\beta^{\vee}}{\lambda} \in \BZ.
\end{equation*}
\end{dfn}
%
%
\begin{dfn} \label{dfn:QLS}
A {quantum Lakshmibai-Seshadri} (QLS for short) {path} of shape $\lambda$ is a pair
%
%
\begin{equation} \label{eq:QLS}
\eta=(x_{1},\,x_{2},\,\dots,\,x_{s}\,;\
\sigma_{0},\,\sigma_{1},\,\dots,\,\sigma_{s})
\end{equation}
of a sequence $x_{1},\,x_{2},\,\dots,\,x_{s}$ of elements 
in $W^{J}$ with $x_{u} \ne x_{u+1}$ for $1 \le u \le s-1$ and 
a sequence $0 = \sigma_{0} < \sigma_{1} < \cdots < \sigma_{s} = 1$ 
of rational numbers satisfying the condition that there exists 
a directed path from $x_{u+1}$ to $x_{u}$ in $\QB_{\sigma_{u}\lambda}(W^{J})$ 
for each $1 \le u \le s-1$; we denote this $x_u \blarrl{\sigma_u\lambda} x_{u+1}$. 
Let $\QLS(\lambda)$ denote the set of all QLS paths of shape $\lambda$. 
\end{dfn}
%
%
\begin{rem} \label{rem:QLS}
We identify $\eta \in \QLS(\lambda)$ of the form \eqref{eq:QLS} 
with the following piecewise-linear, continuous map 
$\eta:[0,1] \rightarrow \BR \otimes_{\BZ} X$: 
%
%
\begin{equation} \label{eq:path}
\eta(t)=\sum_{p=1}^{u-1}
(\sigma_{p}-\sigma_{p-1})x_{p}\lambda+ 
(t-\sigma_{u-1})x_{u}\lambda \quad 
\text{for $\sigma_{u-1} \le t \le \sigma_{u}$, $1 \le u \le s$}.
\end{equation}
In \cite[Theorem~3.3]{LNSSS2}, we proved that 
$\QLS(\lambda)$ is identical (as a set of piecewise-linear, 
continuous maps from $[0,1]$ to $\BR \otimes_{\BZ} X$) 
to the set $\BB(\lambda)_{\cl}$ of ``projected'' Lakshmibai-Seshadri 
paths of shape $\lambda$; for the definition 
of $\BB(\lambda)_{\cl}$, see \cite[\S2.2]{LNSSS2}. 
\end{rem}

Let $\eta = (x_{1},\,\dots,\,x_{s}\,;\,\sigma_{0},\,\sigma_{1},\,\dots,\,\sigma_{s}) 
\in \QLS(\lambda)$. We define the weight $\wt(\eta)$ of $\eta \in \QLS(\lambda)$ by 
%
%
\begin{equation} \label{eq:wt}
\wt (\eta) := \eta(1) =\sum_{u=1}^{s}(\sigma_{u}-\sigma_{u-1})x_{u}\lambda; 
\end{equation}
we can show in exactly the same way as \cite[Lemma~4.5\,a)]{L} that 
$\wt(\eta) \in X$. 
Also, we define the degree $\Deg(\eta)$ as follows
(see \cite[\S4.2 and Theorem~4.6]{LNSSS2}). First, let $x,\,y \in W^{J}$, and let 
\begin{equation*}
x=y_{0} \edge{\beta_{1}} y_{1} \edge{\beta_{2}} \cdots 
\edge{\beta_{k}} y_{k}=y
\end{equation*}
be a shortest directed path from $x$ to $y$ in $\QB(W^{J})$. 
Then we set
\begin{equation}\label{defwtlam}
\wt_{\lambda}(x \Rightarrow y):=
\sum_{ 
   \begin{subarray}{c}
   1 \le p \le k \\[1.5mm] 
   \text{$y_{p-1} \stackrel{\beta_{p}}{\rightarrow} y_{p}$ is a quantum edge} 
   \end{subarray}
} \pair{\beta_{p}^{\vee}}{\lambda} \in \BZ_{\ge 0};
\end{equation}
we see from \cite[Proposition~4.1]{LNSSS2} that 
this value does not depend on the choice of a shortest directed path 
from $x$ to $y$ in  $\QB(W^{J})$. For $\eta = 
(x_{1},\,\dots,\,x_{s}\,;\,\sigma_{0},\,\sigma_{1},\,\dots,\,\sigma_{s}) 
\in \QLS(\lambda)$, we define
%
%
\begin{equation} \label{eq:Deg}
\Deg(\eta):=-\sum_{u=1}^{s-1} (1-\sigma_{u}) 
   \wt_{\lambda}(x_{u+1} \Rightarrow x_{u}) \in \BZ_{\le 0}. 
\end{equation}

For $\eta = (x_{1},\,\dots,\,x_{s}\,;\,\sigma_{0},\,\sigma_{1},\,\dots,\,\sigma_{s}) 
\in \QLS(\lambda)$, we set $\iota(\eta):=x_{1} \in W^{J}$, and 
call it the initial direction of $\eta$. 
Now, for each $w \in W^{J}$, we set
\begin{equation}
\QLS_{w}(\lambda):=\bigl\{ \eta \in \QLS(\lambda) \mid 
\iota(\eta) \le w \bigr\},
\end{equation}
and define the graded character 
$\gch \QLS_{w}(\lambda)$ of $\QLS_{w}(\lambda) \subset \QLS(\lambda)$ by
\begin{equation*}
\gch \QLS_{w}(\lambda) := \sum_{\eta \in \QLS_{w}(\lambda)} 
  q^{-\Deg(\eta)} e^{\wt (\eta)}. 
\end{equation*}
We will prove that for each $w \in W^{J}$, the equality
\begin{equation}
\gch \QLS_{w}(\lambda) = \Mac{w\lambda}
\end{equation}
holds, where $\Mac{w\lambda}$ denotes the specialization of 
the nonsymmetric Macdonald polynomial $\Mact{w\lambda}$ at $t=0$. 

%
\subsection{Orr-Shimozono formula.}
\label{subsec:OS}

In this subsection, we review a formula (\cite[Corollary~4.4]{OS}) for 
the specialization at $t=0$ of nonsymmetric Macdonald polynomials. 

Let $\ti{\Fg}$ denote the dual Lie algebra of $\Fg$, and 
let $\bigl\{\ti{\alpha}_{i}\bigr\}_{i \in I}$ and 
$\bigl\{\ti{\alpha}_{i}^{\vee}\bigr\}_{i \in I}$ be 
the simple roots and the simple coroots of $\ti{\Fg}$, respectively. 
We denote by $\ti{W}$ the Weyl group of $\ti{\Fg}$; 
note that $W \cong \ti{W}$. As is well-known, 
for $w \in W \cong \ti{W}$ and $i \in I$, 
\begin{equation} \label{eq:iden2}
w \ti{\alpha}_{i} = \sum_{j \in I} c_{j}\ti{\alpha}_{j}
\quad \text{if and only if}  \quad
w \alpha_{i}^{\vee} =
\sum_{j \in I} c_{j}\alpha_{j}^{\vee}.
\end{equation}
Hence we identify $w\ti{\alpha}_{i}$ with $w\alpha_{i}^{\vee}$ 
for $w \in W \cong \ti{W}$ and $i \in I$: 
\begin{equation} \label{eq:iden}
w \ti{\alpha}_{i}
\quad \stackrel{\text{identify}}{\longleftrightarrow} \quad
w \alpha_{i}^{\vee}. 
\end{equation}
Let $\ti{\Phi}^{+}$ denote the set of positive roots of $\ti{\Fg}$, 
which we identify with the set $\Phi^{\vee+}$ of 
positive coroots of $\Fg$ by \eqref{eq:iden}. 

Now, let $\ti{\Fg}_{\af}$ denote 
the untwisted affine Lie algebra associated to $\ti{\Fg}$.
Let $\bigl\{\ti{\alpha}_{i}\bigr\}_{i \in I_{\af}}$ 
be the simple roots of $\ti{\Fg}_{\af}$, 
where $I_{\af}=I \sqcup \{0\}$, 
and $\ti{\delta}$ the null root of $\ti{\Fg}_{\af}$. 
We denote by $\ti{\Phi}_{\af}^{+}$ (resp., $\ti{\Phi}_{\af}^{-}$)
the set of positive (resp., negative) real roots of $\ti{\Fg}_{\af}$; 
note that 
\begin{equation*}
\ti{\Phi}_{\af}^{+}=
 \bigl(\underbrace{\BZ_{\ge 0}\ti{\delta}+\ti{\Phi}^{+}}_{%
   \begin{subarray}{c}
   \text{identified with} \\[1mm]
   \BZ_{\ge 0}\ti{\delta}+\Phi^{\vee+}
   \end{subarray}
 } \bigr) \sqcup 
 \bigl(\underbrace{\BZ_{> 0}\ti{\delta}-\ti{\Phi}^{+}}_{%
   \begin{subarray}{c}
   \text{identified with} \\[1mm]
   \BZ_{> 0}\ti{\delta}-\Phi^{\vee+}
   \end{subarray}
 } \bigr).
\end{equation*}
Denote by $\ti{W}_{\af}$ the Weyl group of $\ti{\Fg}_{\af}$; 
note that $\ti{W}_{\af} \cong Q \rtimes \ti{W} \cong Q \rtimes W$. 
Also, we denote by $\ti{W}_{\ext}:=X \rtimes \ti{W} \cong 
X \rtimes W$ the extended affine Weyl group of $\ti{\Fg}_{\af}$, 
and by $t_{\mu} \in \ti{W}_{\ext}$ the translation by $\mu \in X$.
For $x \in \ti{W}_{\ext}$, define $\wt(x) \in X$ and $\dir(x) \in W$ by: 
\begin{equation*}
x = t_{\wt(x)}\dir(x). 
\end{equation*}

For an integral weight $\mu \in X$ for $\Fg$, we set 
\begin{equation*}
m_{\mu}:=t_{\mu}v(\mu)^{-1} \in X \rtimes W \cong \ti{W}_{\ext},
\end{equation*}
where $v(\mu)$ denotes the shortest element in $W$ such that $v(\mu)\mu$ is 
an antidominant integral weight (see \cite[(2.45)]{OS}). 
The following lemma will be used later. 
%
%
\begin{lem} \label{lem:mxi}
Let $\lambda \in X$ be a dominant integral weight, and let $w \in W^{J}$, 
where $J=J_{\lambda}=\bigl\{i \in I \mid \pair{\alpha_{i}^{\vee}}{\lambda}=0\bigr\}$. 
Then, $v(w\lambda)= \mcr{\lng}w^{-1}$, and hence
\begin{equation*}
m_{w\lambda}=t_{w\lambda} (\mcr{\lng}w^{-1})^{-1} = 
w (\mcr{\lng})^{-1} t_{\lng\lambda}. 
\end{equation*}
In particular, 
\begin{equation*}
\begin{cases}
v(\lamm)=v(\mcr{\lng}\lambda)=e, \qquad m_{\lamm}=t_{\lamm}, \\[1.5mm]
v(\lambda) = \mcr{\lng}, \qquad m_{\lambda} = 
(\mcr{\lng})^{-1} t_{\lng\lambda}, 
\end{cases}
\end{equation*}
and $m_{w\lambda}=wm_{\lambda}$. 
\end{lem}

\begin{proof}
It is obvious that $(\mcr{\lng}w^{-1})w\lambda 
= \lng\lambda$ is antidominant. Hence it suffices to show that 
$\ell(x) \ge \ell( \mcr{\lng}w^{-1} )$ 
for all $x \in W$ such that $xw\lambda=\lng\lambda$. 
If $xw\lambda=\lng\lambda$, then $\lng xw \in W_{J}$, and hence 
$x=\lng zw^{-1}$ for some $z \in W_{J}$; note that 
$\ell( zw^{-1} ) = \ell(wz^{-1}) = \ell(w) + \ell(z^{-1})$ since 
$w \in W^{J}$ and $z \in W_{J}$. Therefore, 
\begin{equation*}
\ell(x) = \ell(\lng) - \ell(zw^{-1}) = 
\ell(\lng) - \ell(w) - \ell(z^{-1}). 
\end{equation*}
Here we remark that $\mcr{\lng} = \lng \lngJ$, 
where $\lngJ \in W_{J}$ is the longest element. 
Hence it follows from the computation above 
(with $z$ replaced by $\lngJ$) that 
\begin{equation*}
\ell(\mcr{\lng}w^{-1}) = 
\ell(\lng \lngJ w^{-1}) = 
\ell(\lng) - \ell(w) - \ell(\lngJ^{-1}). 
\end{equation*}
Since $\ell(z^{-1}) \le \ell(\lngJ^{-1})$, we obtain 
$\ell(x) \ge \ell(\mcr{\lng}w^{-1})$, as desired. 
\end{proof}

We fix an arbitrary $\mu \in X$, and apply the argument in \cite[\S3.3]{OS} 
to the case that $u=e$ (the identity element) and $w=m_{\mu}$; 
we generally follow the notation thereof. Let 
%
%
\begin{equation} \label{eq:redw}
m_{\mu} = \pi 
 \underbrace{ r_{i_1} r_{i_2} \cdots r_{i_{\ell}} }_{\in \ti{W}_{\af}}
\end{equation}
be a reduced expression for $m_{\mu}$, 
where $\pi$ is an (affine) Dynkin diagram 
automorphism of $\ti{\Fg}_{\af}$, and set
%
%
\begin{equation} \label{eq:betak}
\beta_{k}^{\OS}
 := r_{i_{\ell}} \cdots r_{i_{k+1}}\ti{\alpha}_{i_{k}}
\quad \text{for $1 \le k \le \ell$}, 
\end{equation}
which is a positive real root of $\ti{\Fg}_{\af}$ 
contained in $\BZ_{> 0} \ti{\delta} - \ti{\Phi}^{+}$ 
(see \cite[Remark 3.17]{OS}). Then we can write $\beta_{k}^{\OS}$ as: 
%
%
\begin{equation} \label{eq:ak}
\beta_{k}^{\OS}=a_{k}\ti{\delta}+\ol{\beta_{k}^{\OS}}
\quad \text{for $a_{k} \in \BZ_{> 0}$ and 
$\ol{\beta_{k}^{\OS}} \in \ti{\Phi}^{-}$, \quad $1 \le k \le \ell$}; 
\end{equation}
we think of $\ol{\beta_{k}^{\OS}}$ as an element of $\Phi^{\vee-}$
under the identification \eqref{eq:iden} 
of $\ti{\Phi}^{+}$ and $\Phi^{\vee+}$, and set 
$\gamma_{k}^{\OS}:=-(\ol{\beta_{k}^{\OS}})^{\vee} \in \Phi^{+}$. 

Let $A=\bigl\{j_{1} < j_{2} < \cdots < j_{r}\bigr\}$
be a subset of $\bigl\{1,\,2,\,\dots,\,\ell\bigr\}$. 
Following \cite[(3.16) and (3.17)]{OS} 
(recall that $u=e$ and $w = m_{\mu}$), 
we set 
\begin{equation*}
z_{0}:=m_{\mu}, \qquad
z_{k}:=z_{k-1}r_{\beta_{j_{k}}^{\OS}} \quad
  \text{for $1 \le k \le r$}; 
\end{equation*}
or equivalently, $z_{0} = m_{\mu}$, and 
$z_{k}$ is obtained from the reduced expression \eqref{eq:redw} 
by removing the $j_{1}$-th reflection, 
the $j_{2}$-th reflection, $\dots$, 
and the $j_{k}$-th reflection. 
We express these data as: 
%
%
\begin{equation} \label{eq:pJ}
p_{A}=\Bigl(
z_{0} \edge{\beta_{j_1}^{\OS}} 
z_{1} \edge{\beta_{j_2}^{\OS}} \cdots 
      \edge{\beta_{j_r}^{\OS}} z_{r}
\Bigr).
\end{equation}
\begin{dfn}[{\cite[\S4.2]{OS}}] \label{dfn:QBX}
Keep the notation and setting above. 
We say that $p_{A}$ is an element of $\QBX$ if 
\begin{equation*}
\dir(z_{0}) \edge{ \gamma_{j_{1}}^{\OS} } 
\dir(z_{1}) \edge{ \gamma_{j_{2}}^{\OS} } \cdots 
\edge{ \gamma_{j_{r}}^{\OS} } \dir(z_{r})
\end{equation*}
is a directed path in 
the quantum Bruhat graph $\QB(W)=\QB(W^{\emptyset})$ for $W$. 
\end{dfn}

For an element $p_{A} \in \QBX$, 
we set (see \cite[(3.19)]{OS})
%
%
\begin{equation} \label{eq:A-}
A^{-}:=
 \bigl\{j_{k} \in A \mid 
 \text{$\dir(z_{k-1}) \edge{\gamma_{j_k}^{\OS}} \dir(z_{k})$ 
     is a quantum edge}
 \bigr\} \subset A, 
\end{equation}
and then set (see \cite[(4.1)]{OS})
%
%
\begin{equation} \label{eq:qwt}
\qwt (p_{A}) : = \sum_{j \in A^{-}} \beta_{j}^{\OS},
\end{equation}
which is contained in 
$\BZ_{> 0}\ti{\delta} - \ti{Q}^{+}$ if $A^{-} \ne \emptyset$, 
where $\ti{Q}^{+}:=\sum_{i \in I} \BZ_{\ge 0} \ti{\alpha}_{i}$. 
Furthermore, in view of equation \eqref{eq:ak}, we set 
(in the notation of \cite[(2.4)]{OS})
%
%
\begin{equation} \label{eq:degr}
\degr(\qwt (p_{A})):=\sum_{j \in A^{-}} a_{j} \in \BZ_{\ge 0}.
\end{equation}
Also, if $p_{A} \in \QBX$ is of the form \eqref{eq:pJ}, then 
we set 
%
%
\begin{equation} \label{eq:wtp}
\enp{A}:=z_{r} \in \ti{W}_{\ext} = X \rtimes W \quad \text{and} \quad 
\wt(p_{A}): = \wt(\enp{A}).
\end{equation} 
%
%
\begin{thm}[{\cite[Corollary~4.4]{OS}}] \label{thm:OS}
Keep the notation and setting above. We have
\begin{equation*}
\Mac{\mu} = \sum_{p \in \QBX} e^{\wt(p)}q^{\degr(\qwt(p))}.
\end{equation*}
\end{thm}
%
%
\subsection{Bijective correspondence between $\QBM$ and $\CA(-\lamm)$.}
\label{subsec:QB-A}

First, we recall the quantum alcove model from \cite{LL} 
(see also \cite[\S5.1]{LNSSS2}). 
We set 
$H_{\alpha,\,n}:=\bigl\{\zeta \in \Fh^{\ast}_{\BR} \mid 
\pair{\alpha^{\vee}}{\zeta}=n \bigr\}$ for $\alpha \in \Phi$ and $n \in \BZ$, 
where $\Fh^{\ast}_{\BR}:=\BR \otimes_{\BZ} X=
\bigoplus_{i \in I} \BR\alpha_{i}$. 
An alcove is, by definition, a connected component 
(with respect to the usual topology on $\Fh^{\ast}_{\BR}$) of 
\begin{equation*}
\Fh^{\ast}_{\BR} \setminus 
 \bigcup_{\alpha \in \Phi^{+},\,n \in \BZ} H_{\alpha,\,n}.
\end{equation*}
We say that two alcoves are adjacent if they are distinct and 
have a common wall. For adjacent alcoves $A$ and $B$, 
we write $A \edge{\alpha} B$, with $\alpha \in \Phi$, 
if their common wall is contained in the hyperplane
$H_{\alpha,\,n}$ for some $n \in \BZ$, and if $\alpha$ points 
in the direction from $A$ to $B$. An alcove path is a sequence of 
alcoves $(A_{0},\,A_{1},\,\dots,\,A_{s})$ such that $A_{u-1}$ and $A_{u}$ 
are adjacent for each $u=1,\,2,\,\dots,\,s$. We say that 
$(A_{0},\,A_{1},\,\dots,\,A_{s})$ is reduced if it has minimal length
among all alcove paths from $A_{0}$ to $A_{s}$. 

Recall that $\ti{W}_{\ext} \cong X \rtimes W$ 
acts (as affine transformations) on $\Fh^{\ast}_{\BR}$ by
\begin{equation*}
(t_{\xi}w) \cdot \zeta = w\zeta + \xi \qquad 
\text{for $\xi \in X$, $w \in W$, and $\zeta \in \Fh^{\ast}_{\BR}$}. 
\end{equation*}
%
%
\begin{rem} \label{rem:refl}
For $\beta = \alpha^{\vee} + n\ti{\delta} \in \ti{\Phi}_{\af}^{+}$ with 
$\alpha \in \Phi^{+}$ and $n \in \BZ_{\ge 0}$ 
(here we identify $\ti{\Phi}^{+}$ with $\Phi^{\vee+}$ under \eqref{eq:iden}), 
we have
$r_{\alpha^{\vee}+n\ti{\delta}} \cdot \zeta = 
(t_{-n\alpha}r_{\alpha^{\vee}}) \cdot \zeta = 
r_{\alpha^{\vee}}\zeta-n\alpha = r_{\alpha}\zeta-n\alpha$ 
for $\zeta \in \Fh^{\ast}_{\BR}$. 
Hence $r_{\alpha^{\vee}+n\ti{\delta}} \in \ti{W}_{\ext}$ 
acts on $\Fh^{\ast}_{\BR}$ as 
the affine reflection with respect to the hyperplane 
$H_{\alpha,\,-n} = H_{-\alpha,\,n}$. 
\end{rem}

Now, let $\lambda \in X$ be a dominant integral weight; 
note that $\lng\lambda \in X$ is antidominant, 
where $\lng \in W$ denotes the longest element. We set 
\begin{equation*}
A_{\circ}:=\bigl\{
 \zeta \in \Fh^{\ast}_{\BR} \mid 
  \text{$0 < \pair{\alpha^{\vee}}{\zeta} < 1$ for all $\alpha \in \Phi^{+}$}
 \bigr\},
\end{equation*}
and $A_{\lamm}:=A_{\circ}+\lamm$. 
\begin{dfn}
The sequence of roots $(\gamma_{1},\,\gamma_{2},\,\dots,\,\gamma_{\ell})$ is 
called a $(-\lng\lambda)$-chain of roots if 
\begin{equation*}
A_{\circ}=A_{0} \edge{-\gamma_{1}} A_{1} 
\edge{-\gamma_{2}} \cdots \edge{-\gamma_{\ell}} A_{\ell}=A_{\lamm}
\end{equation*}
is a reduced alcove path. 
\end{dfn}

Here we note that $m_{\lamm}=t_{\lamm}$ by Lemma~\ref{lem:mxi}. 
It follows from \cite[Lemma~5.3]{LP1} that 
there exists a bijection: 
%
%
\begin{equation} \label{eq:1to1}
\bigl\{ \text{reduced expressions for $m_{\lamm}=t_{\lamm}$} \bigr\}
\quad \stackrel{\text{1:1}}{\longleftrightarrow} \quad 
\bigl\{ \text{$(-\lamm)$-chains of roots} \bigr\}.
\end{equation}
More precisely, let $m_{\lamm}=t_{\lamm}=\pi r_{i_1}r_{i_2} \cdots r_{i_{\ell}}$ 
be a reduced expression for $m_{\lamm}=t_{\lamm} \in \ti{W}_{\ext}$. We set
$A_{k}:=(\pi r_{i_1}r_{i_2} \cdots r_{i_k}) \cdot A_{\circ}$ 
for $0 \le k \le \ell$, and 
%
%
\begin{equation} \label{eq:betaL0}
\beta_{k}^{\Le}:=\pi r_{i_1} \cdots r_{i_{k-1}}(\ti{\alpha}_{i_k}) = 
r_{\p{1}} \cdots r_{\p{k-1}}(\ti{\alpha}_{\p{k}})
\quad 
\text{for $1 \le k \le \ell$}; 
\end{equation}
note that $\beta_{k}^{\Le}$ is a positive real root of 
$\ti{\Fg}_{\af}$ contained in $\BZ_{\ge 0}\ti{\delta}+\ti{\Phi}^{+}$.
In fact, by \cite[(2.4.7)]{Mac}, we have
\begin{align}
\bigl\{ \beta_{k}^{\Le} \mid 1 \le k \le \ell \bigr\} 
  & = \ti{\Phi}_{\af}^{+} \cap m_{\lamm}\ti{\Phi}_{\af}^{-} 
    = \ti{\Phi}_{\af}^{+} \cap t_{\lamm}\ti{\Phi}_{\af}^{-} \nonumber \\
  & = \bigl\{  b \ti{\delta} + \beta^{\vee} \mid 
      \beta \in \Phi^{+},\,0 \le b < -\pair{\beta^{\vee}}{\lamm} \bigr\} \label{eq:inv}
\end{align}
under the identification \eqref{eq:iden} 
of $\ti{\Phi}^{+}$ and $\Phi^{\vee+}$. 
Therefore, we can write $\beta_{k}^{\Le}$ in the form
%
%
\begin{equation} \label{eq:bk}
\beta_{k}^{\Le} = b_{k}\ti{\delta}+\ol{\beta_{k}^{\Le}}, \ 
\text{
  with $b_{k} \in \BZ_{\ge 0}$ and 
  $\ol{\beta_{k}^{\Le}} \in -\lng\bigl(\Phi^{\vee+} \setminus \Phi^{\vee+}_{J}\bigr)$}, 
\end{equation}
for each $1 \le k \le \ell$. If we set $\gamma_{k}^{\Le}:=
(\ol{\beta_{k}^{\Le}})^{\vee} \in -\lng\bigl(\Phi^{+} \setminus \Phi^{+}_{J}\bigr)$, then
%
%
\begin{equation} \label{eq:chain1}
A_{\circ}=A_{0} \edge{-\gamma_{1}^{\Le}} A_{1} 
\edge{-\gamma_{2}^{\Le}} \cdots \edge{-\gamma_{\ell}^{\Le}} A_{\ell}=A_{\lamm}
\end{equation}
is a $(-\lamm)$-chain of roots. 
%
%
\begin{rem}[{see \cite[\S6.1]{LNSSS2}}] \label{rem:bk}
Let $1 \le k \le \ell$. We see from Remark~\ref{rem:refl} that 
the action of $r_{\beta_{k}^{\Le}} \in \ti{W}_{\af}$ on $\Fh^{\ast}_{\BR}$ is 
the affine reflection with respect to 
the hyperplane $H_{\gamma_{k}^{\Le},\,-b_{k}}$.
Also, we know that 
%
%
\begin{equation} \label{eq:bk2}
0 \le b_{k}=\# \bigl\{ 1 \le p < k \mid \gamma_{p}^{\Le} = \gamma_{k}^{\Le} \bigr\} < 
\pair{\ol{\beta_{k}^{\Le}}}{-\lamm}; 
\end{equation}
the sequence $(b_{1},\,\dots,\,b_{\ell})$ is 
called the height sequence 
for the $(-\lamm)$-chain \eqref{eq:chain1}. 
\end{rem}
%
%
\begin{rem} \label{rem:betaLPOS}
Keep the notation and setting above. 
If we define $\beta_{k}^{\OS}$, $1 \le k \le \ell$, 
by \eqref{eq:betak} for the reduced expression 
$m_{\lamm}=t_{\lamm}=\pi r_{i_1}r_{i_2} \cdots r_{i_{\ell}}$, then we have
$\beta_{k}^{\Le}=-t_{\lamm}(\beta_{k}^{\OS})$ for all $1 \le k \le \ell$. 
In particular, $\ol{\beta_{k}^{\Le}} = - \ol{\beta_{k}^{\OS}}$
(see \eqref{eq:ak} and \eqref{eq:bk}), 
and hence $\gamma_{k}^{\Le}=\gamma_{k}^{\OS}=:\gamma_{k}$. 
Also, we have $b_{k}=\pair{\gamma_{k}^{\vee}}{-\lamm}-a_{k}$.
\end{rem}

Now, let 
%
%
\begin{equation} \label{eq:chain}
A_{\circ}=A_{0} \edge{-\gamma_{1}} A_{1} 
\edge{-\gamma_{2}} \cdots \edge{-\gamma_{\ell}} A_{\ell}=A_{\lamm}
\end{equation}
be a $(-\lamm)$-chain of roots. 
%
%
\begin{dfn} \label{dfn:QA}
Let $\CA(-\lamm)$ denote the set of all subsets 
$A=\bigl\{j_{1} < \cdots < j_{r}\bigr\}$ of 
$\bigl\{1,\,2,\,\dots,\,\ell\bigr\}$ such that 
\begin{equation}\label{deffinal}
e \edge{\gamma_{j_1}} r_{\gamma_{j_1}} 
\edge{\gamma_{j_2}} r_{\gamma_{j_1}}r_{\gamma_{j_2}}
\edge{\gamma_{j_3}} \cdots \edge{\gamma_{j_r}}
r_{\gamma_{j_1}}r_{\gamma_{j_2}} \cdots r_{\gamma_{j_r}}=:\phi(A)
\end{equation}
is a directed path in the quantum Bruhat graph $\QB(W)$ for $W$. 
The subsets $A$ are called admissible subsets, and 
$\phi(A)$ is called the final direction of $A$.
\end{dfn}

For $A=\bigl\{j_{1} < \cdots < j_{r}\bigr\} \in \CA(-\lamm)$, 
we define $\wt(A) \in X$, $\Ht(A) \in \BZ_{\ge 0}$ 
(see \cite[Definition~5.1 and (7.1)]{LNSSS2}), 
and $\cHt(A) \in \BZ_{\ge 0}$ as follows: 
%
%
\begin{equation} \label{eq:wtA}
\begin{split}
\wt(A)
 & :=-r_{\beta_{j_1}^{\Le}}r_{\beta_{j_2}^{\Le}} \cdots r_{\beta_{j_r}^{\Le}} \cdot (\lamm) \\
 & =-r_{\gamma_{j_1}^{\Le},\,-b_{j_1}}r_{\gamma_{j_2}^{\Le},\,-b_{j_2}} \cdots 
     r_{\gamma_{j_r}^{\Le},\,-b_{j_r}} \cdot (\lamm), 
\end{split}
\end{equation}
%
%
\begin{equation} \label{eq:height}
\Ht(A):= \sum_{j \in A_{-}} \Bigl(
 \pair{(\gamma_{j}^{\Le})^{\vee} }{ -\lamm } - b_{j} \Bigr), 
\end{equation}
%
%
\begin{equation}\label{defcoheight}
\cHt(A):= \sum_{j \in A_{-}}  b_{j}, 
\end{equation}
where 
%
%
\begin{equation} \label{eq:a-}
A_{-}:=
\bigl\{ j_{k} \in A \mid 
\text{
  $r_{\gamma_{j_1}} \cdots r_{\gamma_{j_{k-1}}} \edge{ \gamma_{j_k}^{\Le} } 
   r_{\gamma_{j_1}} \cdots r_{\gamma_{j_{k-1}}}r_{\gamma_{j_{k}}}$
is a quantum edge} \bigr\}.
\end{equation}

Let $m_{\lamm}=t_{\lamm}=\pi r_{i_1}r_{i_2} \cdots r_{i_{\ell}}$ be 
the reduced expression for $m_{\lamm}=t_{\lamm}$ corresponding 
to the $(-\lamm)$-chain of roots \eqref{eq:chain} under 
the correspondence \eqref{eq:1to1}. 
We define $\QBM$ by using this reduced expression 
for $m_{\lamm}=t_{\lamm}$. Note that 
%
%
\begin{equation} \label{eq:gamma}
\gamma_{k} = \gamma_{k}^{\Le} = \gamma_{k}^{\OS} \qquad \text{for $1 \le k \le \ell$}. 
\end{equation}
%
%
\begin{lem} \label{lem:A-QB}
Keep the notation and setting above. 
Then, 
\begin{equation*}
A \in \CA(-\lamm) \quad \text{\rm if and only if} \quad p_{A} \in \QBM.
\end{equation*}
Hence we have a bijection from $\CA(-\lamm)$ onto 
$\QBM$ that maps $A \in \CA(-\lamm)$ to $p_{A} \in \QBM$. 
Moreover, we have
%
%
\begin{equation} \label{eq:HtWt}
\wt(A) = - \wt(p_{A}) \quad \text{\rm and} \quad 
\Ht(A)=\deg(\qwt(p_{A})) \quad \text{\rm for all $A \in \CA(-\lamm)$}. 
\end{equation}
\end{lem}

\begin{proof}
Let $A=\bigl\{j_{1} < \cdots < j_{r}\bigr\}$. Then, we have
\begin{align*}
& p_{A} \in \QBM \iff 
\underbrace{\dir(z_0)}_{=e} 
\edge{\gamma_{j_1}^{\OS}}
\dir(z_1) \edge{\gamma_{j_2}^{\OS}}
   \cdots \edge{\gamma_{j_r}^{\OS}}
\dir(z_r) \quad \text{in $\QB(W)$} \\[3mm]
& \quad \iff 
e \edge{\gamma_{j_1}} r_{\gamma_{j_1}} 
\edge{\gamma_{j_2}} \cdots \edge{\gamma_{j_r}}
r_{\gamma_{j_1}}r_{\gamma_{j_2}} \cdots r_{\gamma_{j_r}}
\quad \text{in $\QB(W)$ by \eqref{eq:gamma}} \\
& \quad \iff A \in \CA(-\lamm).
\end{align*}

Next, we prove that $\Ht(A)=\deg(\qwt(p_{A}))$ for all $A \in \CA(-\lamm)$.
Let $A=\bigl\{j_{1} < \cdots < j_{r}\bigr\} \in \CA(-\lamm)$; 
we see from the argument above that the set $A^{-}$ in \eqref{eq:A-} is 
identical to the set $A_{-}$ in \eqref{eq:a-}. Then, we see that
\begin{align*}
\Ht(A) & = \sum_{j \in A_{-}} \Bigl(
 \pair{(\gamma_{j}^{\Le})^{\vee} }{ -\lamm } - b_{j} \Bigr) 
 \quad \text{by definition \eqref{eq:height}} \\[3mm]
& = \sum_{j \in A^{-}} \Bigl(
 \underbrace{\pair{\gamma_{j}^{\vee} }{ -\lamm } - b_{j}}_{=a_{j}} \Bigr) 
 \quad \text{by Remark~\ref{rem:betaLPOS}} \\[3mm]
& = \sum_{j \in A^{-}} a_{j} = \deg(\qwt(p_{A}))
 \quad \text{by \eqref{eq:degr}}. 
\end{align*}

Finally, we show that $\wt(A) = - \wt(p_{A})$ 
for all $A \in \CA(-\lamm)$; 
we proceed by induction on 
the cardinality of $A \in \CA(-\lamm)$. 
First, observe that this equality is obvious if $A=\emptyset$. 
Now, let us take $A=\bigl\{j_1 < \cdots < j_{r-1} < j_{r} \bigr\} \in \CA(-\lamm)$, 
and set $A':=\bigl\{j_1 < \cdots < j_{r-1}\bigr\}$, 
which is also an element of $\CA(-\lamm)$.
By direct computation, together with definition \eqref{eq:wtA}, 
we can show that 
%
%
\begin{equation} \label{eq:wt1}
\wt (A) = \wt (A')-
\bigl( \pair{ \gamma_{j_{r}}^{\vee} }{-\lamm}-b_{j_{r}} \bigr) 
r_{\gamma_{j_1}} \cdots r_{\gamma_{j_{r-1}}} (\gamma_{j_{r}});
\end{equation}
or, we may refer the reader to the proof of \cite[Proposition~6.7]{LNSSS2}. 
Also, we have
\begin{equation*}
z_{r} = z_{r-1} r_{\beta_{j_{r}}^{\OS}} = 
  z_{r-1}r_{ a_{j_r}\ti{\delta}+\ol{\beta_{j_r}^{\OS}} } = 
  z_{r-1} \bigl( t_{ -a_{j_r}(\ol{\beta_{j_r}^{\OS}})^{\vee} } 
  r_{ \ol{\beta_{j_r}^{\OS}} } \bigr) = 
  z_{r-1} t_{ a_{j_r}\gamma_{j_{r}} }r_{ \gamma_{j_{r}} }.
\end{equation*}
Therefore, if we write $z_{r}=t_{\wt (z_{r})}\dir(z_{r})$ and 
$z_{r-1}=t_{\wt (z_{r-1})}\dir(z_{r-1})$, then we deduce that
\begin{align*}
t_{\wt(z_{r})}\dir(z_{r}) 
 & = t_{\wt(z_{r-1})}\dir(z_{r-1}) t_{ a_{j_r}\gamma_{j_{r}} }r_{ \gamma_{j_{r}} }
   = t_{\wt(z_{r-1})}t_{ a_{j_r}\dir(z_{r-1})\gamma_{j_{r}} } 
     \dir(z_{r-1}) r_{ \gamma_{j_{r}} } \\
 & = t_{\wt(z_{r-1}) + a_{j_r}\dir(z_{r-1})\gamma_{j_{r}} } 
     \bigl(\dir(z_{r-1}) r_{ \gamma_{j_{r}} }\bigr),
\end{align*}
and hence
\begin{equation*}
\wt(p_{A})=\wt(z_{r})=\wt(z_{r-1}) + a_{j_r}\dir(z_{r-1})\gamma_{j_{r}}.
\end{equation*}
Here, since $a_{j_{r}}=\pair{\gamma_{j_{r}}^{\vee}}{-\lamm}-b_{j_{r}}$ by 
Remark~\ref{rem:betaLPOS}, we obtain 
\begin{align*}
\wt(p_{A}) & = \wt(z_{r-1}) + 
 \bigl(\pair{\gamma_{j_{r}}^{\vee}}{-\lamm}-b_{j_{r}}\bigr)
 \dir(z_{r-1})\gamma_{j_{r}} \\
& = \wt(p_{A'})+  \bigl(\pair{\gamma_{j_{r}}^{\vee}}{-\lamm}-b_{j_{r}}\bigr)
 \dir(z_{r-1})\gamma_{j_{r}}; 
\end{align*}
note that $\dir(z_{r-1})=r_{\gamma_{j_{1}}} \cdots r_{\gamma_{j_{r-1}}}$ since 
$\dir(z_{0})=\dir(m_{\lamm})=e$. Hence it follows that
\begin{align*}
 \wt (p_{A}) & = 
 \wt(p_{A'})+  \bigl(\pair{\gamma_{j_{r}}^{\vee}}{-\lamm}-b_{j_{r}}\bigr)
 r_{\gamma_{j_{1}}} \cdots r_{\gamma_{j_{r-1}}}(\gamma_{j_{r}}) \\
& = -\wt (A') + \bigl(\pair{\gamma_{j_{r}}^{\vee}}{-\lamm}-b_{j_{r}}\bigr)
 r_{\gamma_{j_{1}}} \cdots r_{\gamma_{j_{r-1}}}(\gamma_{j_{r}}) \\
& \hspace*{50mm} \text{by our induction hypothesis} \\
& = -\wt (A) \quad \text{by \eqref{eq:wt1}},
\end{align*}
as desired. This completes the proof of the lemma.
\end{proof}
%
%
\subsection{Lexicographic (lex) $(-\lamm)$-chains of roots.}
\label{subsec:lex}

We keep the notation and setting of the previous subsection; 
we fix a dominant integral weight $\lambda \in X$, and set 
$J=J_{\lambda}=\bigl\{i \in I \mid \pair{\alpha_{i}^{\vee}}{\lambda}=0\bigr\}$. 
For $w \in W$, we simply write $\mcr{w}^{J}=\mcr{w}^{J_{\lambda}} \in W^{J}$ as 
$\mcr{w}$, unless stated otherwise explicitly. 

In \cite[\S4]{LP2} (see also \cite[Proposition~5.4]{LNSSS2}), 
we introduced a specific $(-\lamm)$-chain of roots, called 
a lexicographic (lex for short) $(-\lamm)$-chain of roots. 
We will frequently make use of the following property 
of a lex $(-\lamm)$-chain of roots: If $m_{\lamm}=\pi 
r_{i_1}r_{i_2} \cdots r_{i_{\ell}}$ is the reduced expression for 
$m_{\lamm}$ corresponding to a lex $(-\lamm)$-chain of roots 
(recall from \eqref{eq:1to1} the one-to-one correspondence between 
the reduced expressions for $m_{\lamm}=t_{\lamm}$ and 
the $(-\lamm)$-chains of roots), then we have 
%
%
\begin{equation} \label{eq:bkle}
0 \le \frac{b_{1}}{ \bpair{ \ol{\beta_{1}^{\Le}} }{-\lamm} } \le 
\frac{b_{2}}{ \bpair{ \ol{\beta_{2}^{\Le}} }{-\lamm} } \le \cdots 
\le \frac{b_{\ell}}{ \bpair{ \ol{\beta_{\ell}^{\Le}} }{-\lamm}} < 1, 
\end{equation} 
where $\beta_{k}^{\Le} = b_{k}\ti{\delta} + \ol{\beta_{k}^{\Le}}$ 
for $1 \le k \le \ell$ is given as in \eqref{eq:betaL0} and \eqref{eq:bk} 
(see also Remark~\ref{rem:bk}). 

We know from Lemma~\ref{lem:mxi} that 
$m_{\lamm} = t_{\lamm} = \mcr{\lng} (t_{\lambda} \mcr{\lng}^{-1})$ and 
$m_{\lambda} = t_{\lambda} \mcr{\lng}^{-1}$. It follows that 
$m_{\lamm} = \mcr{\lng}m_{\lambda}$. 
Also, since $\ell(t_{\lambda})= \ell(m_{\lambda} \mcr{\lng} ) = 
\ell(m_{\lambda}) + \ell(\mcr{\lng})$ by \cite[(2.4.5)]{Mac}, 
we have 
%
%
\begin{equation} \label{eq:elltmu}
\ell(m_{\lamm}) = \ell(t_{\lamm})= \ell(t_{\lambda}) = \ell(m_{\lambda})+\ell(\mcr{\lng});
\end{equation}
note that $\ell(t_{\lamm})=\ell(t_{\lambda})$ by \cite[(2.4.1)]{Mac}. 
This implies that the concatenation of a reduced expression for $\mcr{\lng}$ with
a reduced expression for $m_{\lambda}$ is 
a reduced expression for $m_{\lamm}=t_{\lamm}$. 
We set $M:=\ell(\mcr{\lng})$. 
%
%
\begin{lem} \label{lem:lex}
Let $m_{\lamm}=\pi r_{i_1}r_{i_2} \cdots r_{i_{\ell}}$ be 
the reduced expression for $m_{\lamm}$ corresponding to a lex $(-\lamm)$-chain of roots under 
the correspondence \eqref{eq:1to1}. Then, 
\begin{equation*} 
\mcr{\lng} = r_{\p{1}} \cdots r_{\p{M}} \quad \text{\rm and} \quad
m_{\lambda} = \pi r_{i_{M+1}} \cdots r_{i_{\ell}}.
\end{equation*}
Namely, 
\begin{equation*}
m_{\lamm}=\pi r_{i_1}r_{i_2} \cdots r_{i_{\ell}}=
( \underbrace{r_{\p{1}} \cdots r_{\p{M}}}_{=\mcr{\lng} \ \text{\rm (reduced)}} ) 
( \underbrace{\pi r_{i_{M+1}} \cdots r_{i_{\ell}}}_{=m_{\lambda} \ \text{\rm (reduced)}}). 
\end{equation*}
\end{lem}

\begin{proof}
We make use of \eqref{eq:bkle}. 
Let $K$ be the maximal index such that 
$b_{K} / \pair{ \ol{\beta_{K}^{\Le}} }{-\lamm}  = 0$. 
Then we see that 
\begin{equation} \label{eq:lex1a}
\bigl\{ \beta_{k}^{\Le} \mid 1 \le k \le \ell \bigr\} \cap 
  \Bigl(-\lng(\Phi^{\vee+} \setminus \Phi^{\vee+}_{J})\Bigr) =
\bigl\{ \beta_{k}^{\Le} \mid 1 \le k \le K \bigr\}. 
\end{equation}
Also, we see from \eqref{eq:inv} that 
$-\lng(\Phi^{\vee+} \setminus \Phi^{\vee+}_{J}) \subset
\bigl\{\beta_{k}^{\Le} \mid 1 \le k \le \ell \bigr\}$. Hence
the left-hand side of \eqref{eq:lex1a} is identical to 
$-\lng(\Phi^{\vee+} \setminus \Phi^{\vee+}_{J})$. 
From these, by noting that 
$\#(\Phi^{\vee+} \setminus \Phi^{\vee+}_{J}) = 
\ell(\lng)-\ell(\lngJ)=\ell(\mcr{\lng}) = M$ 
(recall that $\lngJ$ is the longest element of $W_{J}$), 
we conclude that $K=M$, and hence that 
$\bigl\{ \beta_{k}^{\Le} \mid 1 \le k \le M \bigr\} = 
-\lng(\Phi^{\vee+} \setminus \Phi^{\vee+}_{J})$. 
In addition, since 
%
%
\begin{equation} \label{eq:btLe1}
\beta_{k}^{\Le} 
 = \pi r_{i_{1}} \cdots r_{i_{k-1}}(\ti{\alpha}_{i_{k}})
 = r_{\p{1}} \cdots r_{\p{k-1}}(\ti{\alpha}_{\p{k}}) \in
 -\lng(\Phi^{\vee+} \setminus \Phi^{\vee+}_{J})
\end{equation}
for all $1 \le k \le M$, we see easily that $\p{1},\,\dots,\,\p{M} \in I$. 

We will show that $v:=r_{\p{1}} \cdots r_{\p{M}} \in W$ 
is identical to $\mcr{\lng}$. By the argument above, we have
\begin{equation*}
\bigl\{\alpha \in \Phi^{\vee +} \mid 
v^{-1}\alpha \in \Phi^{\vee-}\bigr\} = 
\bigl\{ \beta_{k}^{\Le} \mid 1 \le k \le M \bigr\} = 
-\lng(\Phi^{\vee+} \setminus \Phi^{\vee+}_{J}). 
\end{equation*}
From this, we see that 
%
%
\begin{equation} \label{eq:lexA}
-v^{-1}\lng(\Phi^{\vee+} \setminus \Phi^{\vee+}_{J}) \subset \Phi^{\vee-}, 
   \quad \text{so that} \quad
v^{-1}\lng(\Phi^{\vee+} \setminus \Phi^{\vee+}_{J}) \subset \Phi^{\vee+}.
\end{equation}
Hence it follows that $\bigl\{\alpha \in \Phi^{\vee +} \mid 
v^{-1}\lng\alpha \in \Phi^{\vee-}\bigr\} \subset \Phi^{\vee+}_{J}$.
Since $v=r_{\p{1}} \cdots r_{\p{M}}$ is a reduced expression, 
we have $\ell(v)=M$, and hence 
\begin{equation*}
\# \bigl\{\alpha \in \Phi^{\vee +} \mid 
v^{-1}\lng\alpha \in \Phi^{\vee-}\bigr\} = \ell(v^{-1}\lng)=N-M.
\end{equation*}
Also, we have $\# \Phi^{\vee+}_{J} = \ell(\lngJ) = \ell(\lng)-\ell(\mcr{\lng})=N-M$. 
Therefore, we deduce that
%
%
\begin{equation} \label{eq:lexB}
\bigl\{\alpha \in \Phi^{\vee +} \mid 
v^{-1}\lng\alpha \in \Phi^{\vee-}\bigr\} = \Phi^{\vee+}_{J}.
\end{equation}
Since $\lngJ(\Phi^{\vee+} \setminus \Phi^{\vee+}_{J}) \subset
\Phi^{\vee+} \setminus \Phi^{\vee+}_{J}$ and 
$\lngJ(\Phi^{\vee+}_{J}) \subset \Phi^{\vee-}_{J}$, we have
\begin{align*}
& v^{-1}\lng \lngJ (\Phi^{\vee+} \setminus \Phi^{\vee+}_{J}) 
  \subset v^{-1}\lng (\Phi^{\vee+} \setminus \Phi^{\vee+}_{J}) 
  \subset \Phi^{\vee+} \qquad \text{by \eqref{eq:lexA}}, \\
& v^{-1}\lng \lngJ (\Phi^{\vee+}_{J}) 
  \subset v^{-1}\lng (\Phi^{\vee-}_{J}) 
  \subset \Phi^{\vee+}_{J} \qquad \text{by \eqref{eq:lexB}}. 
\end{align*}
From these, we obtain $v^{-1}\lng \lngJ (\Phi^{\vee+}) \subset \Phi^{\vee+}$, 
which implies that $v^{-1}\lng \lngJ = e$, and hence that 
$v=\lng\lngJ = \mcr{\lng}$, as desired. 
Finally, because $\ell(m_{\lambda})=\ell(m_{\lamm})-\ell(\mcr{\lng})=\ell-M$
and $m_{\lamm}=\mcr{\lng}m_{\lambda}$, it follows that
$m_{\lambda} = \pi r_{i_{M+1}} \cdots r_{i_{\ell}}$ is 
a reduced expression for $m_{\lambda}$. 
This proves the lemma.
\end{proof}

Fix a lex $(-\lamm)$-chain of roots. 
We construct $\QBM$ from the reduced expression 
$m_{\lamm}=\pi r_{i_1}r_{i_2} \cdots r_{i_{\ell}}$ 
corresponding to the lex $(-\lamm)$-chain of roots under \eqref{eq:1to1}, 
which we denote by $\QBM_{\lex}$; recall from 
\eqref{eq:betak}, \eqref{eq:ak}, and Remark~\ref{rem:betaLPOS} 
that for $1 \le k \le \ell$, 
\begin{equation*}
\begin{cases}
\beta_{k}^{\OS}
  = r_{i_{\ell}} \cdots r_{i_{k+1}}\ti{\alpha}_{i_{k}} 
  = a_{k}\ti{\delta} + \ol{\beta_{k}^{\OS}}, & 
  \text{with $a_{k} \in \BZ_{> 0}$ and $\ol{\beta_{k}^{\OS}} \in \ti{\Phi}^{-}$}, \\[1.5mm]
\gamma_{k}=\gamma_{k}^{\OS}=-(\ol{\beta_{k}^{\OS}})^{\vee} \in \Phi^{+}, \\[1.5mm]
b_{k}=\pair{\gamma_{k}^{\vee}}{-\lng\lambda}-a_{k}.
\end{cases}
\end{equation*}
We see from \eqref{eq:btLe1} that 
%
%
\begin{equation} \label{eq:gam}
\gamma_{k}=\gamma_{k}^{\Le} = (\beta_{k}^{\Le})^{\vee} = 
r_{\p{1}} \cdots r_{\p{k-1}}(\alpha_{\p{k}}) \qquad 
 \text{for $1 \le k \le M=\ell(\mcr{\lng})$}, 
\end{equation}
and hence 
%
%
\begin{equation} \label{eq:ga2}
\bigl\{\gamma_{1},\,\dots,\,\gamma_{M}\bigr\} = 
\bigl\{(\beta_{1}^{\Le})^{\vee},\,\dots,\,(\beta_{M}^{\Le})^{\vee}\bigr\} = 
 -\lng(\Phi^{+} \setminus \Phi^{+}_{J}) = \Phi^{+} \setminus \Phi^{+}_{\omega(J)}, 
\end{equation}
where $\omega:I \rightarrow I$ is the Dynkin diagram automorphism given by: 
$\lng \alpha_{i} = - \alpha_{\omega(i)}$ for $i \in I$. Also, it follows from 
the equality ``$K=M$'' (shown in the proof of Lemma~\ref{lem:lex}), 
together with \eqref{eq:bkle}, that 
%
%
\begin{equation} \label{eq:bkle2}
0 = \frac{b_{1}}{ \pair{ \gamma_{1}^{\vee} }{-\lamm} } = \cdots = 
    \frac{b_{M}}{ \pair{ \gamma_{M}^{\vee} }{-\lamm} } < 
    \frac{b_{M+1}}{ \pair{ \gamma_{M+1}^{\vee} }{-\lamm} } \le \cdots \le 
    \frac{b_{\ell}}{ \pair{ \gamma_{\ell}^{\vee} }{-\lamm}} < 1. 
\end{equation} 

Now, let $\mcr{\lng}=r_{p_1}r_{p_2} \cdots r_{p_M}$ be 
an (arbitrary) reduced expression for $\mcr{\lng}$, 
and set $i_{k}':=\pi^{-1}(p_{k})$ for $1 \le k \le M$. 
We see from Lemma~\ref{lem:lex} that 
%
%
\begin{equation} \label{eq:red-mlamm}
m_{\lamm}=
( \underbrace{r_{p_{1}} \cdots r_{p_{M}}}_{=\mcr{\lng}} ) 
( \underbrace{\pi r_{i_{M+1}} \cdots r_{i_{\ell}}}_{=m_{\lambda}})=
\pi r_{i_1'} \cdots r_{i_M'}r_{i_{M+1}} \cdots r_{i_{\ell}} 
\end{equation}
is a reduced expression for $m_{\lamm}$, which we denote by $R$. 
We construct $\QBM$ from this reduced expression $R$ of $m_{\lamm}$, 
and denote it by $\QBM_{R}$. Then, 
\begin{align*}
\beta_{k}^{\OS, R} & := 
\begin{cases}
\underbrace{r_{i_{\ell}} \cdots r_{i_{M+1}}}_{=m_{\lambda}^{-1}\pi} 
  r_{i_{M}'} \cdots r_{i_{k+1}'}\ti{\alpha}_{i_{k}'} = 
  m_{\lambda}^{-1}r_{p_{M}} \cdots r_{p_{k+1}}\ti{\alpha}_{p_{k}} & 
\text{\rm for $1 \le k \le M$}, \\[5mm]
r_{i_{\ell}} \cdots r_{i_{k+1}}\ti{\alpha}_{i_{k}} = \beta_{k}^{\OS} & 
\text{\rm for $M+1 \le k \le \ell$}, 
\end{cases} \\[5mm]
& = a_{k}^{R}\ti{\delta} + \ol{\beta_{k}^{\OS, R}} \quad 
 \text{for some $a_{k}^{R} \in \BZ_{> 0}$ and $\ol{\beta_{k}^{\OS, R}} \in \ti{\Phi}^{-}$}, 
\end{align*}
\begin{equation*}
\gamma_{k}^{\OS, R} : = 
 - (\ol{\beta_{k}^{\OS, R}})^{\vee} \in \Phi^{+}. 
\end{equation*}
Also, for the reduced expression $R$ of $m_{\lamm}$ in \eqref{eq:red-mlamm}, 
we define $\beta_{k}^{\Le, R}$, $1 \le k \le \ell$, as in \eqref{eq:betaL0}, 
and write it as: 
$\beta_{k}^{\Le, R} = b_{k}^{R}\ti{\delta} + \ol{\beta_{k}^{\Le, R}}$, 
with $b_{k}^{R} \in \BZ_{\ge 0}$ and $\ol{\beta_{k}^{\Le, R}} \in 
-\lng(\Phi^{\vee+} \setminus \Phi^{\vee+}_{J})$ (see \eqref{eq:bk}). 
Then we set $\gamma_{k}^{\Le, R} : = (\ol{\beta_{k}^{\Le, R}})^{\vee}$ 
for $1 \le k \le \ell$. By Remark~\ref{rem:betaLPOS}, we have
\begin{equation*}
\gamma_{k}^{\Le, R} = \gamma_{k}^{\OS, R}=:\gamma_{k}^{R}
\quad \text{and} \quad
b_{k}^{R}=\pair{(\gamma_{k}^{R})^{\vee}}{-\lng \lambda} - a_{k}^{R}
\quad \text{for $1 \le k \le \ell$}. 
\end{equation*}
Notice that $\beta_{k}^{\Le, R} = r_{p_{1}} \cdots r_{p_{k-1}}(\ti{\alpha}_{p_{k}})$
for $1 \le k \le M$. Since $p_{1},\,\dots,\,p_{M} \in I$, we see that 
$\beta_{k}^{\Le, R} \in \Phi^{\vee+}$ for all $1 \le k \le M$, 
which implies that $b_{k}^{R} = 0$ and 
%
%
\begin{equation} \label{eq:gamR}
\gamma_{k}^{R} =
\gamma_{k}^{\Le, R} = 
  (\ol{\beta_{k}^{\Le, R}})^{\vee} = 
  (\beta_{k}^{\Le, R})^{\vee}=
  r_{p_1} \cdots r_{p_{k-1}}(\alpha_{p_{k}})
\end{equation}
for all $1 \le k \le M$.

%
\begin{lem} \label{lem:R}
Keep the notation and setting above. We have
\begin{align}
& \bigl\{\beta^{\OS,R}_{k} \mid 1 \le k \le M\bigr\} = \bigl\{\beta^{\OS}_{k} \mid 1 \le k \le M\bigr\}, 
  \label{eq:R1} \\
& \beta^{\OS,R}_{k} = \beta^{\OS}_{k} \qquad \text{\rm for all $M+1 \le k \le \ell$}.
  \label{eq:R2}
\end{align}
Hence
\begin{align}
& \bigl\{\gamma^{R}_{k} \mid 1 \le k \le M\bigr\} 
  = \bigl\{\gamma_{k} \mid 1 \le k \le M\bigr\} 
  = \Phi^{+} \setminus \Phi^{+}_{\omega(J)}, \label{eq:R3} \\
& \gamma^{R}_{k} = \gamma_{k} \quad \text{\rm for all $M+1 \le k \le \ell$}, \label{eq:R4} \\[1.5mm]
& b_{k}^{R}=
  \begin{cases}
  0 & \text{\rm for $1 \le k \le M$}, \\[1.5mm]
  b_{k} > 0 & \text{\rm for $M+1 \le k \le \ell$}. 
  \end{cases} \label{eq:R5}
\end{align}
\end{lem}

\begin{proof}
It is obvious from the definitions that 
$\beta^{\OS,R}_{k} = \beta^{\OS}_{k}$ for all $M+1 \le k \le \ell$. 
We see from these equalities and \eqref{eq:bkle2} that 
\begin{equation*}
\gamma_{k}^{R} = \gamma_{k} \quad \text{and} \quad
b_{k}^{R}=b_{k} > 0 \quad \text{for all $M+1 \le k \le \ell$}.
\end{equation*}
Also, we have shown that $b_{k}^{R} = 0$ for all $1 \le k \le M$ 
(see the comment preceding this lemma). 

It remains to show \eqref{eq:R1} and \eqref{eq:R3}. 
Since $\mcr{\lng}=r_{p_1} \cdots r_{p_M} = r_{\p{1}} \cdots r_{\p{M}}$ 
are reduced expressions, it follows that 
\begin{align*}
\bigl\{r_{p_{M}} \cdots r_{p_{k+1}}(\ti{\alpha}_{p_{k}}) \mid 1 \le k \le M\bigr\} 
& = \bigl\{ \ti{\alpha} \in \Phi^{\vee+} \mid \mcr{\lng}\ti{\alpha} \in -\Phi^{\vee+} \bigr\} \\
& = \bigl\{r_{\p{M}} \cdots r_{\p{k+1}}(\ti{\alpha}_{\p{k}}) \mid 1 \le k \le M\bigr\};
\end{align*}
notice that 
%
%
\begin{equation} \label{eq:invw0}
\bigl\{ \ti{\alpha} \in \Phi^{\vee+} \mid \mcr{\lng}\ti{\alpha} \in -\Phi^{\vee+} \bigr\} = 
 \Phi^{\vee+} \setminus \Phi_{J}^{\vee+}.
\end{equation}
Indeed, we see from \eqref{eq:mcrs} that 
$\bigl\{ \ti{\alpha} \in \Phi^{\vee+} \mid \mcr{\lng}\ti{\alpha} \in -\Phi^{\vee+} \bigr\} 
 \subset \Phi^{\vee+} \setminus \Phi^{\vee+}_{J}$. Conversely, 
 if $\ti{\alpha} \in \Phi^{\vee+} \setminus \Phi^{\vee+}_{J}$, then 
$\lngJ \ti{\alpha} \in \Phi^{\vee+}$, and hence 
$\mcr{\lng} \ti{\alpha} = \lng \lngJ \ti{\alpha} \in -\Phi^{\vee+}$, as desired. 
Therefore, we deduce from the definitions that
\begin{align*}
\bigl\{\beta^{\OS,R}_{k} \mid 1 \le k \le M\bigr\} 
 & = m_{\lambda}^{-1}(\Phi^{\vee +} \setminus \Phi^{\vee+}_{J}) 
   = \bigl\{\beta^{\OS}_{k} \mid 1 \le k \le M\bigr\}, 
\end{align*}
and hence that
\begin{equation*}
 \bigl\{\gamma^{R}_{k} \mid 1 \le k \le M\bigr\} 
  = \bigl\{\gamma_{k} \mid 1 \le k \le M\bigr\} 
  \stackrel{\eqref{eq:ga2}}{=} \Phi^{+} \setminus \Phi_{\omega(J)}^{+}.
\end{equation*}
This proves the lemma. 
\end{proof}

We set $\ell(\lng):=N$; since $\lng = \mcr{\lng} \lngJ$, it follows that 
$\ell(\lngJ) = N-M$. Fix a reduced expression 
$\lngJ = r_{t_{M+1}} r_{t_{M+2}} \cdots r_{t_{N}}$ for $\lngJ$. Then,
\begin{equation*}
\lng = 
 \ub{r_{\p{1}} \cdots r_{\p{M}}}{=\mcr{\lng}} 
 \ub{r_{t_{M+1}} r_{t_{M+2}} \cdots r_{t_{N}}}{=\lngJ} =
 \ub{r_{p_1} \cdots r_{p_M}}{=\mcr{\lng}} 
 \ub{r_{t_{M+1}} r_{t_{M+2}} \cdots r_{t_{N}}}{=\lngJ}
\end{equation*}
are reduced expressions for $\lng$. Now we set
\begin{equation} \label{eq:gamM}
\xi_{k}=\xi_{k}^{R}:=\mcr{\lng}r_{t_{M+1}} \cdots r_{t_{k-1}}\alpha_{t_{k}} \in \Phi^{+}
\qquad \text{for $M+1 \le k \le N$}.
\end{equation}
Then, by \eqref{eq:gam} and \eqref{eq:gamR}, 
both of the sets 
$\bigl\{\gamma_{k} \mid 1 \le k \le M \bigr\} \cup 
\bigl\{\xi_{k} \mid M+1 \le k \le N \bigr\}$ and 
$\bigl\{\gamma_{k}^{R} \mid 1 \le k \le M \bigr\} \cup 
\bigl\{\xi_{k}^{R} \mid M+1 \le k \le N \bigr\}$ 
are identical to $\Phi^{+}$. 
Hence it follows from \eqref{eq:R3} that
%
%
\begin{equation} \label{eq:gamM2}
\bigl\{\xi^{R}_{k} \mid M+1 \le k \le N\bigr\} 
  = \bigl\{\xi_{k} \mid M+1 \le k \le N\bigr\} 
  = \Phi^{+}_{\omega(J)}. 
\end{equation}
If we define total orders $\prec$ and $\prec_{R}$ on $\Phi^{+}$ by: 
\begin{align}
 & \ub{\gamma_{1} \prec \cdots \prec \gamma_{M}}
      {\in \Phi^{+} \setminus \Phi_{\omega(J)}^{+}} \prec 
   \ub{\xi_{M+1} \prec \cdots \prec \xi_{N}}{\in \Phi_{\omega(J)}^{+}}, 
   \label{eq:prec} \\[3mm]
 & \ub{\gamma_{1}^{R} \prec_{R} \cdots \prec_{R} \gamma_{M}^{R}}
      {\in \Phi^{+} \setminus \Phi_{\omega(J)}^{+}} \prec 
   \ub{\xi_{M+1}^{R} \prec_{R} \cdots \prec_{R} \xi_{N}^{R}}{\in \Phi_{\omega(J)}^{+}}, 
   \label{eq:precR}
\end{align}
respectively, then these total orders are reflection orders 
(see, for example, \cite[Chap.\,5, Exerc.\,20]{BB}). 

Let $A=\bigl\{j_{1},\,j_{2},\,\dots,\,j_{r}\bigr\} \subset \bigl\{1,\,2,\,\dots,\,\ell\bigr\}$ 
be such that 
\begin{equation*}
\begin{split}
& p_{A}=\Bigl(
m_{\lamm} = t_{\lamm} = z_{0} \edge{\beta_{j_1}^{\OS}} \cdots 
      \edge{\beta_{j_r}^{\OS}} z_{r}
\Bigr) \in  \QBM_{\lex}, \\[3mm]
& \text{(resp.,\ }
p_{A}=\Bigl(
m_{\lamm} = t_{\lamm} = z_{0}^{R} \edge{\beta_{j_1}^{\OS,R}} \cdots 
      \edge{\beta_{j_r}^{\OS,R}} z_{r}^{R}
\Bigr) \in  \QBM_{R} \text{ )}; 
\end{split}
\end{equation*}
we set $j_{0}:=0$ by convention. 
By the definition (see Definition~\ref{dfn:QBX}), 
we have a directed path 
\begin{equation*}
\begin{split}
& 
e=\dir(z_{0}) \edge{ \gamma_{j_{1}} } \cdots 
\edge{ \gamma_{j_{r}} } \dir(z_{r}) \\[3mm]
& 
\text{(resp.,\ }
e=\dir(z_{0}^{R}) \edge{ \gamma_{j_{1}}^{R} } \cdots 
\edge{ \gamma_{j_{r}}^{R} } \dir(z_{r}^{R})
\text{)}
\end{split}
\end{equation*}
in the quantum Bruhat graph $\QB(W)$. Let us take 
$0 \le s \le r$ such that $j_{s} \le M$ and $j_{s+1} \ge M+1$, 
and set 
%
%
\begin{equation} \label{eq:tii}
\begin{split}
& \ti{\iota}(p_{A}):=\dir(z_{s}) =r _{\gamma_{j_1}} \cdots r _{\gamma_{j_s}} \in W \\
& \text{(resp., \ }
\ti{\iota}(p_{A}):=\dir(z_{s}^{R}) =r _{\gamma_{j_1}^{R}} \cdots r _{\gamma_{j_s}^{R}} \in W\text{)}.
\end{split}
\end{equation}
%
%
\begin{rem} \label{rem:dirz}
Because $\gamma_{j_1} \prec \gamma_{j_2} \prec \cdots \prec \gamma_{j_s}$ 
with respect to the reflection order $\prec$ on $\Phi^{+}$ (see \eqref{eq:prec}), 
we deduce from \cite[Theorem~6.4]{LNSSS2} that 
$e= \dir(z_{0}) \edge{ \gamma_{j_{1}} } \cdots \edge{ \gamma_{j_{s}} } \dir(z_{s})=\ti{\iota}(p_{A})$
is a shortest directed path from $e$ to $\ti{\iota}(p_{A})$ 
in the quantum Bruhat graph $\QB(W)$. Therefore, all the edges in this directed path 
are Bruhat edges by Remark~\ref{rem:QB}\,(3). 
We show by induction on $u$ that 
$\dir(z_{u}) \in W^{\omega(J)}$ for all $0 \le u \le s$. 
If $u=0$, then it is obvious that $\dir(z_{0}) = e \in W^{\omega(J)}$. 
Assume that $0 < u \le s$. Since $\dir(z_{u-1}) \in W^{\omega(J)}$ 
by our induction hypothesis, and 
since $\dir(z_{u-1}) \edge{ \gamma_{j_{u}} } 
\dir(z_{u}) = \dir(z_{u-1})r_{\gamma_{j_u}}$ is a Bruhat edge in $\QB(W)$, 
we see by \cite[Corollary~2.5.2]{BB} that 
$\dir(z_{u}) \in W^{\omega(J)}$ or 
$\dir(z_{u}) = \dir(z_{u-1})r_{i}$ for some $i \in \omega(J)$. 
Suppose that $\dir(z_{u}) = \dir(z_{u-1})r_{i}$ for some $i \in \omega(J)$. 
Since $\dir(z_{u-1})r_{\gamma_{j_u}} = \dir(z_{u}) = \dir(z_{u-1})r_{i}$, 
we have $r_{\gamma_{j_u}} = r_{i}$, and hence $\gamma_{j_u} = 
\alpha_{i} \in \Phi_{\omega(J)}^{+}$, which contradicts the fact that
$\gamma_{j_u} \in \Phi^{+} \setminus \Phi_{\omega(J)}^{+}$. 
Thus we obtain $\dir(z_{u}) \in W^{\omega(J)}$, as desired. 
In particular, $\ti{\iota}(p_{A}) = \dir (z_{s}) \in W^{\omega(J)}$. 
The same argument works also for the reduced expression $R$. 
\end{rem}

Here we define a map $\Theta_{R}^{\lex} : \QBM_{R} \rightarrow \QBM_{\lex}$. 
Let $A=\bigl\{j_{1},\,j_{2},\,\dots,\,j_{r}\bigr\} \subset \bigl\{1,\,2,\,\dots,\,\ell\bigr\}$ 
be such that 
%
%
\begin{equation} \label{eq:pA}
p_{A}=\Bigl(
m_{\lamm} = t_{\lamm} = 
  z_{0}^{R} 
      \edge{\beta_{j_1}^{\OS,R}} \cdots 
      \edge{\beta_{j_r}^{\OS,R}} 
  z_{r}^{R}
\Bigr) \in  \QBM_{R},
\end{equation}
that is, 
%
%
\begin{equation} \label{eq:Theta0}
e=\dir(z_{0}^{R}) \edge{ \gamma_{j_{1}}^{R} } 
\dir(z_{1}^{R}) \edge{ \gamma_{j_{2}}^{R} } \cdots 
\edge{ \gamma_{j_{r}}^{R} } \dir(z_{r}^{R})
\end{equation}
is a directed path in the quantum Bruhat graph $\QB(W)$. 
If we take $0 \le s \le r$ such that 
$j_{s} \le M$ and $j_{s+1} \ge M+1$, then we have a shortest directed path
\begin{equation*}
e=\dir(z_{0}^{R}) \edge{ \gamma_{j_{1}}^{R} } 
\dir(z_{1}^{R}) \edge{ \gamma_{j_{2}}^{R} } \cdots 
\edge{ \gamma_{j_{s}}^{R} } \dir(z_{s}^{R}) = \ti{\iota}(p_{A}) \in W^{\omega(J)}
\end{equation*}
in the quantum Bruhat graph $\QB(W)$; 
note that $\gamma_{j_{1}}^{R} \prec_{R} \cdots \prec_{R} \gamma_{j_{s}}^{R}$ 
with respect to the reflection order $\prec_{R}$ on $\Phi^{+}$ (see \eqref{eq:precR}).
We know from \cite[Theorem~6.4]{LNSSS2} that 
there exists a unique shortest directed path 
\begin{equation*}
e = x_{0} \edge{ \gamma_{q_1} } \cdots \edge{ \gamma_{q_u} } x_{u} 
\edge{ \xi_{q_{u+1}} } \cdots \edge{ \xi_{q_{s}} }x_{s} = \ti{\iota}(p_{A})
\end{equation*}
from $e$ to $\ti{\iota}(p_{A})$ in $\QB(W)$ such that 
$1 \le q_1 < \cdots < q_u \le M < q_{u+1} \le \cdots \le q_{s} \le N = \ell(\lng)$ 
(see \eqref{eq:gam} and \eqref{eq:gamM}) for some $0 \le u \le s$; note that 
all the edges in this directed path are Bruhat edges 
by Remark~\ref{rem:QB}\,(3). 
We claim that $u=s$. Indeed, suppose for a contradiction that $u < s$; 
in this case, $\xi_{q_{s}} \in \Phi_{\omega(J)}^{+}$ by \eqref{eq:gamM2}, 
and hence $r_{\xi_{q_s}} \in W_{\omega(J)}$. 
We write $x_{s-1} = \mcr{x_{s-1}}^{\omega(J)} z$ for some $z \in W_{\omega(J)}$; 
note that $\ell(x_{s-1}) = \ell(\mcr{x_{s-1}}^{\omega(J)}) + \ell(z)$. 
We see that 
\begin{equation*}
\ti{\iota}(p_{A}) = x_{s} = x_{s-1} r_{\xi_{q_s}} = 
\ub{\mcr{x_{s-1}}^{\omega(J)}}{\in W^{\omega(J)}}
\ub{zr_{\xi_{q_s}}}{\in W_{\omega(J)}},
\end{equation*}
and hence $\ell(x_{s}) = \ell(\mcr{x_{s-1}}^{\omega(J)}) + \ell(zr_{\xi_{q_s}})$. 
Because $x_{s-1} \edge{ \xi_{q_s} } x_{s}$ is a Bruhat edge in $\QB(W)$ as seen above, we have 
$\ell(x_{s})=\ell(x_{s-1})+1$. Combining these equalities, we obtain 
\begin{equation*} 
\ell(\mcr{x_{s-1}}^{\omega(J)}) + \ell(zr_{\xi_{q_s}}) = 
\ell(x_{s}) = \ell(x_{s-1})+1 = 
\ell(\mcr{x_{s-1}}^{\omega(J)}) + \ell(z) + 1, 
\end{equation*}
and hence $\ell(zr_{\xi_{q_s}}) = \ell(z) + 1 \ge 1$. 
Hence it follows that $zr_{\xi_{q_s}} \ne e$, 
which implies that $\ti{\iota}(p_{A}) = x_{s} =
\mcr{x_{s-1}}^{\omega(J)} zr_{\xi_{q_s}} \notin W^{\omega(J)}$. 
However, this contradicts the fact that $\ti{\iota}(p_{A}) \in W^{\omega(J)}$ 
(see Remark~\ref{rem:dirz}). Thus, we obtain $u=s$, and hence a directed path
%
%
\begin{equation} \label{eq:Theta1}
e = x_{0} \edge{ \gamma_{q_1} } \cdots \edge{ \gamma_{q_s} } x_{s} 
  = \ti{\iota}(p_{A})
\end{equation}
such that $1 \le q_1 < \cdots < q_s \le M$. 

Now, we set $B:=\bigl\{q_{1},\,\dots,\,q_{s},\,j_{s+1},\,\dots,\,j_{r}\bigr\}$, 
and consider 
%
%
\begin{equation} \label{eq:pB}
p_{B}=\Bigl(
m_{\lamm} = t_{\lamm} = z_{0} 
\edge{\beta_{q_1}^{\OS}} \cdots \edge{\beta_{q_s}^{\OS}} 
z_{s} \edge{\beta_{j_{s+1}}^{\OS}} \cdots 
      \edge{\beta_{j_r}^{\OS}} z_{r}
\Bigr).
\end{equation}
Since $M+1 \le j_{s+1} < \cdots < j_{r} \le \ell$, 
we see from \eqref{eq:R4} that $\gamma_{j_{u}} = \gamma_{j_{u}}^{R}$ 
for all $s+1 \le u \le r$. Therefore, by replacing the first $s$ edges in 
\eqref{eq:Theta0} with \eqref{eq:Theta1}, we obtain a directed path 
\begin{equation*}
e=\dir(z_{0}) \edge{ \gamma_{q_{1}} } \cdots 
\edge{ \gamma_{q_{s}} } \dir(z_{s}) = \ti{\iota}(p_{A}) 
\edge{ \gamma_{j_{s+1}} } \cdots \edge{ \gamma_{j_{r}} } \dir(z_{r})
\end{equation*}
in the quantum Bruhat graph $\QB(W)$. 
Hence we conclude that $p_{B} \in \QBM_{\lex}$; 
we set $\Theta_{R}^{\lex}(p_{A}):=p_{B}$. 
%
%
\begin{prop} \label{prop:Theta}
The map $\Theta_{R}^{\lex} : \QBM_{R} \rightarrow \QBM_{\lex}$ is bijective. 
Moreover, for every $p \in \QBM_{R}$, 
\begin{equation*}
\wt (\Theta_{R}^{\lex}(p)) = \wt (p), \qquad
\qwt (\Theta_{R}^{\lex}(p)) = \qwt (p), \qquad
\ti{\iota}(\Theta_{R}^{\lex}(p)) = \ti{\iota}(p).
\end{equation*}
\end{prop}

\begin{proof}
Let $A=\bigl\{j_{1},\,j_{2},\,\dots,\,j_{r}\bigr\}$ and 
$B = \bigl\{q_{1},\,\dots,\,q_{s},\,j_{s+1},\,\dots,\,j_{r}\bigr\}$ 
be as in the definition above of the map $\Theta_{R}^{\lex}$. 
Recall that 
\begin{equation*}
\begin{split}
& p_{A} =\Bigl(
m_{\lamm}=z_{0}^{R} \edge{\beta_{j_1}^{\OS,R}} \cdots 
\edge{\beta_{j_s}^{\OS,R}} z_{s}^{R} \edge{\beta_{j_{s+1}}^{\OS,R}}
\cdots \edge{\beta_{j_r}^{\OS,R}} z_{r}^{R} = \enp{A}
\Bigr), \\[3mm]
&
e=\dir(z_{0}^{R}) \edge{ \gamma_{j_{1}}^{R} } \cdots 
\edge{ \gamma_{j_{s}}^{R} } \dir(z_{s}^{R}) = \ti{\iota}(p_{A}) 
\edge{ \gamma_{j_{s+1}}^{R} } \cdots \edge{ \gamma_{j_{r}}^{R} } \dir(z_{r}^{R}),
\end{split}
\end{equation*}
and that
\begin{equation*}
\begin{split}
& p_{B} =\Bigl(
m_{\lamm}=z_{0} \edge{\beta_{q_1}^{\OS}} \cdots 
\edge{\beta_{q_s}^{\OS}} z_{s} \edge{\beta_{j_{s+1}}^{\OS}} \cdots 
      \edge{\beta_{j_r}^{\OS}} z_{r}^{R} = \enp{B}
\Bigr), \\[3mm]
& 
e=\dir(z_{0}) \edge{ \gamma_{q_{1}} } \cdots 
\edge{ \gamma_{q_{s}} } \dir(z_{s}) = \ti{\iota}(p_{A}) 
\edge{ \gamma_{j_{s+1}} } \cdots \edge{ \gamma_{j_{r}} } \dir(z_{r}). 
\end{split}
\end{equation*}
Since $q_{s} \le M$ and $j_{s+1} \ge M+1$, it follows that 
$\ti{\iota}(p_{B})=\dir(z_{s})=\ti{\iota}(p_{A})$. 

Next, we prove that 
%
%
\begin{equation} \label{eq:Theta}
\wt (p_{B}) = \wt (p_{A}) \quad \text{and} \quad
\qwt (p_{B}) = \qwt (p_{A}). 
\end{equation}
Recall from \eqref{eq:A-} and \eqref{eq:qwt} that 
\begin{equation*}
\qwt (p_{B}) = \sum_{j \in B^{-}} \beta_{j}^{\OS} 
\quad \text{and} \quad
\qwt (p_{A}) = \sum_{j \in A^{-}} \beta_{j}^{\OS,R}. 
\end{equation*}
We know from Remark~\ref{rem:dirz} that 
all the edges in 
$e=\dir(z_{0}^{R}) \edge{ \gamma_{j_{1}}^{R} } \cdots 
\edge{ \gamma_{j_{s}}^{R} } \dir(z_{s}^{R}) = \ti{\iota}(p_{A})$ and 
$e=\dir(z_{0}) \edge{ \gamma_{q_{1}} } \cdots 
\edge{ \gamma_{q_{s}} } \dir(z_{s}) = \ti{\iota}(p_{A})$ 
are Bruhat edges, which implies that 
$A^{-},\,B^{-} \subset \bigl\{j_{s+1},\,\dots,\,j_{r}\bigr\}$. 
Since $M+1 \le j_{s+1} < \cdots < j_{r} \le \ell$, 
we see from \eqref{eq:R4} that 
$\gamma_{j_{u}}^{R} = \gamma_{j_{u}}$ for all $s+1 \le u \le r$. 
Therefore, the directed paths 
$\dir(z_{s}^{R}) = \ti{\iota}(p_{A}) 
\edge{ \gamma_{j_{s+1}}^{R} } \cdots \edge{ \gamma_{j_{r}}^{R} } \dir(z_{r}^{R})$ 
and $\dir(z_{s}) = \ti{\iota}(p_{A}) 
\edge{ \gamma_{j_{s+1}} } \cdots \edge{ \gamma_{j_{r}} } \dir(z_{r})$ are identical, 
which implies that $A^{-}=B^{-}$. Since $\beta_{j_{u}}^{\OS, R} = \beta_{j_{u}}^{\OS}$ 
for all $s+1 \le u \le r$ by \eqref{eq:R2}, we obtain $\qwt (p_{B}) = \qwt (p_{A})$. 

Finally, we prove that $\wt (p_{B}) = \wt (p_{A})$; it suffices to show that 
$\enp{B}=\enp{A}$ (see \eqref{eq:wtp}). Since $b_{k}^{R}=b_{k}=0$ 
for all $1 \le k \le M$ by \eqref{eq:bkle2} and \eqref{eq:R5}, we see that
\begin{equation*}
\beta_{k}^{\OS,R}= 
  \pair{(\gamma_{k}^{R})^{\vee}}{-\lng \lambda} \ti{\delta} - 
  (\gamma_{k}^{R})^{\vee}, \qquad
\beta_{k}^{\OS}= 
  \pair{\gamma_{k}^{\vee}}{-\lng \lambda} \ti{\delta} - 
  \gamma_{k}^{\vee}
\end{equation*}
for $1 \le k \le M$, which implies that
\begin{equation*}
r_{\beta_{k}^{\OS,R}} 
  = \bigl(t_{ \pair{(\gamma_{k}^{R})^{\vee}}{-\lng \lambda} \gamma_{k}^{R} }\bigr)
    r_{\gamma_{k}^{R}}, \qquad
r_{\beta_{k}^{\OS}} 
  = \bigl(t_{ \pair{\gamma_{k}^{\vee}}{-\lng \lambda} \gamma_{k}}\bigr)
    r_{\gamma_{k}}.
\end{equation*}
Using these equalities, together with $z_{0}=z_{0}^{R}=m_{\lamm} = t_{\lamm}$, 
we can show by induction on $0 \le u \le s$ that 
\begin{equation*}
z_{u}^{R}= t_{ \dir (z_{u}^{R}) \lng \lambda} \dir (z_{u}^{R}), \qquad
z_{u}=t_{ \dir (z_{u}) \lng \lambda} \dir (z_{u}). 
\end{equation*}
Since $\dir (z_{s}^{R}) = \ti{\iota}(p_{A}) = \dir (z_{s})$, we deduce that
\begin{equation*}
z_{s}^{R}= 
t_{ \dir (z_{s}^{R}) \lng \lambda} \dir (z_{s}^{R}) = 
t_{ \dir (z_{s}) \lng \lambda} \dir (z_{s}) = z_{s}. 
\end{equation*}
Since $\beta_{j_{u}}^{\OS, R} = \beta_{j_{u}}^{\OS}$ 
for all $s+1 \le u \le r$ as seen above, we obtain
\begin{equation*}
\enp{A} = z_{s}^{R}r_{\beta_{j_{s+1}}^{\OS,R}} \cdots r_{\beta_{j_{r}}^{\OS,R}} 
= z_{s}r_{\beta_{j_{s+1}}^{\OS}} \cdots r_{\beta_{j_{r}}^{\OS}} = \enp{B}. 
\end{equation*}
This proves \eqref{eq:Theta}. 

If we define a map $\Theta_{\lex}^{R} : \QBM_{\lex} \rightarrow \QBM_{R}$ 
in exactly the same manner as for the map $\Theta_{R}^{\lex}:\QBM_{R} \rightarrow \QBM_{\lex}$, 
then from the uniqueness of a directed path in $\QB(W)$
whose labels are increasing in a reflection order (see \cite[Theorem~6.4]{LNSSS2}), 
we deduce that both of the composites 
$\Theta_{\lex}^{R} \circ \Theta_{R}^{\lex}$ and 
$\Theta_{R}^{\lex} \circ \Theta_{\lex}^{R}$ are the identity maps. 
This proves the bijectivity of the map $\Theta_{R}^{\lex}$, 
and hence completes the proof of the proposition. 
\end{proof}
%
%
\subsection{Embedding of $\QBw$ into $\QBM_{\lex}$.}
\label{subsec:QBL-QBM}

We keep the notation and setting of the previous subsection. 
Recall that $m_{\lamm} = t_{\lamm}= \pi r_{i_{1}}r_{i_{2}} \cdots r_{i_{\ell}}$ 
is the reduced expression for $m_{\lamm} = t_{\lamm}$ corresponding to 
the (fixed) lex $(-\lamm)$-chain of roots; we know from Lemma~\ref{lem:lex} that 
%
%
\begin{equation} \label{eq:mlamm2}
m_{\lamm}=t_{\lamm}=\pi r_{i_{1}}r_{i_{2}} \cdots r_{i_{\ell}} = 
 \bigl( \ub{ r_{\p{1}} \cdots r_{\p{M}} }{=\mcr{\lng}} \bigr)
 \bigl( \ub{ \pi r_{i_{M+1}} \cdots r_{i_{\ell}} }{=m_{\lambda}} \bigr),
\end{equation}
where $M=\ell(\mcr{\lng})$. 

Let $w \in W^{J}$, and set $L:=\ell(w) \le M$. 
We can take a reduced expression 
$\mcr{\lng} = r_{p_1}r_{p_2} \cdots r_{p_M}$ for $\mcr{\lng}$ 
such that $w=r_{p_{M-L+1}} \cdots r_{p_{M}}$;
%
%
\begin{equation} \label{eq:re}
\mcr{\lng} = \ub{r_{p_1} \cdots r_{p_{M-L}}}{=\mcr{\lng}w^{-1}}
\ub{r_{p_{M-L+1}} \cdots r_{p_{M}}}{=w}.
\end{equation}
Indeed, recall that $\lng = \mcr{\lng} \lngJ$, 
with $\ell(\lng) = \ell(\mcr{\lng}) + \ell(\lngJ)$. 
Since $w \in W^{J}$, we have 
$\ell(w \lngJ) = \ell(w) + \ell(\lngJ)$. Hence it follows that
\begin{align*}
\ell(\mcr{\lng}) + \ell(\lngJ)  
& = \ell(\lng) = \ell(\lng (w\lngJ)^{-1}) + \ell(w \lngJ) \\
& = \ell(\mcr{\lng}w^{-1}) + \ell(w) + \ell(\lngJ),
\end{align*}
so that $\ell(\mcr{\lng}) = \ell(\mcr{\lng}w^{-1}) + \ell(w)$, 
which implies that $\ell(\mcr{\lng}w^{-1}) = M-L$. 
Therefore, if $\mcr{\lng}w^{-1} = r_{p_1} \cdots r_{p_{M-L}}$ is 
a reduced expression for $\mcr{\lng}w^{-1}$, and 
$w= r_{p_{M-L+1}} \cdots r_{p_{M}}$ is a reduced expression for $w$, 
then $\mcr{\lng} = r_{p_1} \cdots r_{p_{M-L}}
r_{p_{M-L+1}} \cdots r_{p_{M}}$ is a reduced expression for $\mcr{\lng}$. 
Now, we set $i_{k}':=\pi^{-1}(p_{k})$ for $1 \le k \le M$; 
we see that 
%
%
\begin{equation} \label{eq:mlamm3}
m_{\lamm} = 
 \bigl( \underbrace{ r_{p_1}r_{p_2} \cdots r_{p_M} }_{=\mcr{\lng}} \bigr)
 \bigl( \underbrace{ \pi r_{i_{M+1}} \cdots r_{i_{\ell}} }_{=m_{\lambda}} \bigr) 
 = \pi r_{i_{1}'} \cdots r_{i_{M}'} r_{i_{M+1}} \cdots r_{i_{\ell}}
\end{equation}
is a reduced expression for $m_{\lamm}$. As in \S\ref{subsec:lex}, 
we construct $\QBM$ from this reduced expression $R$ for $m_{\lamm}$, and 
denote it by $\QBM_{R}$; recall from Proposition~\ref{prop:Theta} the bijection 
$\Theta_{R}^{\lex}:\QBM_{R} \rightarrow \QBM_{\lex}$. 
We set $A_{0}:=\bigl\{1,\,2,\,\dots,\,M-L\bigr\} \subset 
\bigl\{1,\,2,\,\dots,\,\ell\bigr\}$, and consider $p_{A_0}$. 
Using Lemma~\ref{lem:mxi}, we see by direct computation that
\begin{align*}
& z_{0}^{R}=m_{\lamm}= 
  \pi r_{i_{1}'} \cdots r_{i_{M}'} r_{i_{M+1}} \cdots r_{i_{\ell}} =t_{\lamm}, \\
& z_{1}^{R}=
  \pi r_{i_{2}'} \cdots r_{i_{M}'} r_{i_{M+1}} \cdots r_{i_{\ell}} = r_{p_1} t_{\lamm}, \\
& z_{2}^{R}=
  \pi r_{i_{3}'} \cdots r_{i_{M}'} r_{i_{M+1}} \cdots r_{i_{\ell}} = r_{p_2} r_{p_1} t_{\lamm}, \\
& \qquad \vdots\\
& z_{M-L-1}^{R}=
  \pi r_{i_{M-L}'} \cdots r_{i_{M}'} r_{i_{M+1}} \cdots r_{i_{\ell}} = 
  r_{p_{M-L-1}} \cdots r_{p_2} r_{p_1} t_{\lamm}, \\
& z_{M-L}^{R}=
  \pi r_{i_{M-L+1}'} \cdots r_{i_{M}'} r_{i_{M+1}} \cdots r_{i_{\ell}} = 
  \underbrace{r_{p_{M-L}} \cdots r_{p_2} r_{p_1}}_{=w\mcr{\lng}^{-1}} t_{\lamm} = m_{w\lambda}. 
\end{align*}
From these, we deduce that $\dir(z_{K}^{R})= r_{p_{K}} \cdots r_{p_{2}}r_{p_1}$ 
for $0 \le K \le M-L$, and that
%
%
\begin{equation} \label{eq:dp5}
e \edge{\gamma_{1}^{R}} r_{p_1}
\edge{\gamma_{2}^{R}} 
\cdots  
\edge{\gamma_{M-L-1}^{R}} r_{p_{M-L-1}} \cdots r_{p_{2}}r_{p_{1}}
\edge{\gamma_{M-L}^{R}} w\mcr{\lng}^{-1}
\end{equation}
is a directed path from $e$ to $w\mcr{\lng}^{-1}$ 
in the quantum Bruhat graph $\QB(W)$; 
since $\ell(\dir(z_{K})) = \ell(\dir(z_{K-1})) + 1$ 
for $1 \le K \le M-L$, all the edges in this directed path 
are Bruhat edges. Hence we obtain
%
%
\begin{equation} \label{eq:pA0}
p_{A_{0}}=
\Bigl(
m_{\lamm} = 
  z_{0}^{R} \edge{\beta_{1}^{\OS,R}} 
  z_{1}^{R} \edge{\beta_{2}^{\OS,R}} \cdots 
      \edge{\beta_{M-L}^{\OS,R}} z_{M-L}^{R}=m_{w\lambda}
\Bigr) \in \QBM_{R}.
\end{equation}

Since $m_{w\lambda}=w \mcr{\lng}^{-1} t_{\lng \lambda} = w m_{\lambda}$ 
by Lemma~\ref{lem:mxi}, we have 
%
%
\begin{equation} \label{eq:mw}
m_{w\lambda} = 
 \bigl( \underbrace{ r_{p_{M-L+1}} \cdots r_{p_M} }_{=w} \bigr)
 \bigl( \underbrace{ \pi r_{i_{M+1}} \cdots r_{i_{\ell}} }_{=m_{\lambda}} \bigr) 
 = \pi r_{i_{M-L+1}'} \cdots r_{i_{M}'} r_{i_{M+1}} \cdots r_{i_{\ell}}; 
\end{equation}
since \eqref{eq:mlamm3} is a reduced expression (for $m_{\lamm}$), 
we see that \eqref{eq:mw} is also a reduced expression (for $m_{w\lambda}$). 
Let us construct $\QBw$ from this reduced expression. 
Namely, for a subset $B=\bigl\{j_{1} < j_{2} < \cdots < j_{r}\bigr\}
\subset \bigl\{M-L+1,\,M-L+2,\,\dots,\,\ell\bigr\}$, we define
\begin{equation*}
y_{0}^{R} = m_{w\lambda}, \qquad
y_{k}^{R} = y_{0}r_{\beta_{j_1}^{\OS,R}} \cdots r_{\beta_{j_{k}}^{\OS,R}} \quad 
 \text{for $1 \le k\le r$}, 
\end{equation*}
where $\beta_{k}^{\OS,R}$, $M-L+1 \le k \le \ell$, are 
those used in the definition of $\QBM_{R}$, and set 
\begin{equation*}
p_{B}:= \Bigl(
m_{w\lambda}=y_{0}^{R} \edge{\beta_{j_1}^{\OS,R}} y_{1}^{R}
\edge{\beta_{j_2}^{\OS,R}} y_{2}^{R} \edge{\beta_{j_3}^{\OS,R}}
\cdots  
\edge{\beta_{j_r}^{\OS,R}} y_{r}^{R}\Bigr).
\end{equation*}
Then, $p_{B} \in \QBw$ if 
\begin{equation*}
w \mcr{\lng}^{-1}=\dir(y_{0}^{R}) \edge{ \gamma_{j_1}^{R} }
\dir(y_{1}^{R}) \edge{ \gamma_{j_2}^{R} } 
\cdots \edge{ \gamma_{j_r}^{R} } \dir(y_{r}^{R})
\end{equation*}
is a directed path in the quantum Bruhat graph $\QB(W)$. 

Since $\enp{A_{0}}=m_{w\lambda}$, 
we can ``concatenate'' $p_{A_{0}}$ with an arbitrary $p_{B} \in \QBw$, 
which is just $p_{A_{0} \sqcup B}$; 
we see easily that $p_{A_{0} \sqcup B} \in \QBM_{R}$. 
%
%
\begin{lem} \label{lem:embed}
There exists an embedding 
$\QBw \hookrightarrow \QBM_{R}$, which maps 
$p_{B} \in \QBw$ to $p_{A_{0} \sqcup B} \in \QBM_{R}$. 
Moreover, 
$\wt(p_{A_{0} \sqcup B}) = \wt (p_{B})$, and 
$\qwt(p_{A_{0} \sqcup B}) = \qwt (p_{B})$ {\rm(}and hence
$\degr(\qwt(p_{A_{0} \sqcup B})) = 
 \degr(\qwt (p_{B}))${\rm)}. 
\end{lem}

\begin{proof}
The injectivity of the map is obvious. 
Since $\enp{A_{0} \sqcup B} = \enp{B}$ by the definition, 
we have $\wt(p_{A_{0} \sqcup B}) = \wt (p_{B})$. 
Because all the edges in the directed path \eqref{eq:dp5} 
are Bruhat edges, we see from the definition \eqref{eq:A-} 
that $(A_{0} \sqcup B)^{-}= B^{-}$. Hence we obtain
$\qwt(p_{A_{0} \sqcup B}) = \qwt (p_{B})$ by the definition 
\eqref{eq:qwt} of $\qwt$. This proves the lemma. 
\end{proof}

We set
\begin{equation*}
\QBM_{R, w}:=
 \bigl\{ p_{A} \in \QBM_{R} \mid 
  \bigl\{1,\,2,\,\dots,\,M-L\bigr\} \subset A
 \bigr\}.
\end{equation*}
We see from the argument above that 
$\QBM_{R, w}$ is identical to the image of 
the embedding $\QBw \hookrightarrow \QBM_{R}$ of Lemma~\ref{lem:embed}. 
%
%
\begin{lem} \label{lem:init}
Let $p_{A} \in \QBM_{R}$. Then, $p_{A} \in \QBM_{R, w}$ if and only if
$\ti{\iota}(p_{A}) \ge w \mcr{\lng}^{-1}$ with respect to 
the Bruhat order $\ge$ on $W$. 
\end{lem}

\begin{proof}
First, we prove the ``only if'' part. 
Since $p_{A} \in \QBM_{R, w}$, it follows that 
$A$ is of the form: $A=\bigl\{1,\,2,\,\dots,\,M-L,\,j_{1},\,\dots,\,j_{r}\bigr\}$
for $M-L+1 \le j_{1} < \cdots < j_{r} \le \ell$; 
we set $j_{0}=0$ by convention. Take $0 \le s \le r$ such that 
$j_{s} \le M$ and $j_{s+1} \ge M+1$. Then, by \eqref{eq:dp5} and 
the definition of $\ti{\iota}(p_{A})$, we have a directed path 
\begin{equation*}
e \edge{\gamma_{1}^{R}} r_{p_1}
\edge{\gamma_{2}^{R}} \cdots \edge{\gamma_{M-L}^{R}} w\mcr{\lng}^{-1} 
\edge{ \gamma_{j_{1}}^{R} } \cdots \edge{ \gamma_{j_{s}}^{R} } \ti{\iota}(p_{A})
\end{equation*}
in the quantum Bruhat graph $\QB(W)$. Since all the edges in this directed path 
are Bruhat edges (see Remark~\ref{rem:dirz}), we obtain 
$\ti{\iota}(p_{A}) \ge w\mcr{\lng}^{-1}$, as desired. 

Next, we prove the ``if'' part. 
Assume that $\ti{\iota}(p_{A}) \ge w \mcr{\lng}^{-1}$, 
with $A=\bigl\{j_{1},\,\dots,\,j_{r}\bigr\} \subset 
\bigl\{1,\,2,\,\dots,\,\ell\bigr\}$. If we take $0 \le s \le \ell$ such that 
$j_{s} \le M$ and $j_{s+1} \ge M+1$, then we have a shortest directed path 
%
%
\begin{equation} \label{eq:dp7}
e=\dir(z_{0}^{R}) \edge{\gamma_{j_1}^{R}} 
  \dir(z_{1}^{R}) \edge{\gamma_{j_2}^{R}} \cdots \edge{\gamma_{j_s}^{R}} 
  \dir(z_{s}^{R}) = \ti{\iota}(p_{A}) 
\end{equation}
in the quantum Bruhat graph $\QB(W)$ all of whose edges are Bruhat edges 
(see Remark~\ref{rem:dirz}); note that $s=\ell(\ti{\iota}(p_{A}))$. 
Here, because $\ti{\iota}(p_{A}) \ge w \mcr{\lng}^{-1}$ 
with respect to the Bruhat order on $W$, we deduce 
by the chain property of the Bruhat order (see \cite[Theorem~2.2.6]{BB}) that 
there exists a directed path $w \mcr{\lng}^{-1} = x_{0} \edge{ \xi_1 } \cdots 
 \edge{ \xi_{s-M+L} } x_{s-M+L} = \ti{\iota}(p_{A})$ 
of length $\ell(\ti{\iota}(p_{A})) - \ell(w \mcr{\lng}^{-1}) = s-(M-L)$ 
from $w \mcr{\lng}^{-1}$ to $\ti{\iota}(p_{A})$ 
in the quantum Bruhat graph $\QB(W)$ all of whose edges are Bruhat edges. 
Concatenating this directed path with the directed path \eqref{eq:dp5}, 
we get the directed path
%
%
\begin{equation} \label{eq:dp6}
e \edge{\gamma_{1}^{R}} r_{p_1}
\edge{\gamma_{2}^{R}} \cdots \edge{\gamma_{M-L}^{R}} w\mcr{\lng}^{-1} 
= x_{0} \edge{ \xi_1 } \cdots \edge{ \xi_{s-M+L} } x_{s-M+L} = \ti{\iota}(p_{A}). 
\end{equation}
Since the length of this directed path is equal to 
$s=\ell(\ti{\iota}(p_{A})) - \ell(e)$, 
this directed path is also a shortest directed path from 
$e$ to $\ti{\iota}(p_{A})$ in the quantum Bruhat graph $\QB(W)$. 
Because the labels in the directed path \eqref{eq:dp7} 
are strictly increasing with respect to the reflection 
order $\prec_{R}$ (see \eqref{eq:precR}), 
that is, $\gamma_{j_1}^{R} \prec_{R} \cdots \prec_{R} \gamma_{j_s}^{R}$, 
it follows from \cite[Theorem~6.4]{LNSSS2} that 
the directed path \eqref{eq:dp7} is lexicographically minimal among
all shortest directed paths from $e$ to $\ti{\iota}(p_{A})$; 
in particular, the directed path \eqref{eq:dp7} is less than 
or equal to the directed path \eqref{eq:dp6}, 
which implies that $j_{1}=1,\,j_{2}=2,\,\dots,\,j_{M-L}=M-L$. 
Thus, we obtain $\bigl\{1,\,2,\,\dots,\,M-L\bigr\} \subset A$, 
and hence $p_{A} \in \QBM_{R, w}$. This completes the proof of the lemma. 
\end{proof}

From Lemma~\ref{lem:init} (together with the comment preceding it), 
Lemma~\ref{lem:embed}, and Proposition~\ref{prop:Theta}, 
we obtain the following proposition. 
%
%
\begin{prop} \label{prop:embed}
The image of $\QBw$ under the composite
\begin{equation*}
\QBw \stackrel{\text{\rm Lemma~\ref{lem:embed}}}{\hookrightarrow} 
\QBM_{R} \stackrel{\Theta_{R}^{\lex}}{\longrightarrow} \QBM_{\lex}
\end{equation*}
is identical to
%
%
\begin{equation} \label{eq:QBMlexw}
\QBM_{\lex, w}:=\bigl\{ p \in \QBM_{\lex} \mid \ti{\iota}(p) \ge w \mcr{\lng}^{-1} \bigr\}.
\end{equation}
Hence we have
%
%
\begin{equation} \label{eq:prf1}
\sum_{p \in \QBM_{\lex,w}} e^{\wt(p)} q^{\degr(\qwt(p))} = 
\sum_{p \in \QBw} e^{\wt(p)} q^{\degr(\qwt(p))} = \Mac{w\lambda}.
\end{equation}
\end{prop}
%
%
\subsection{Bijection between $\QBM_{\lex}$ and $\QLS(\lambda)$.}
\label{subsec:QB-QLS}

As in the previous subsection, we fix a lex $(-\lamm)$-chain of roots
%
%
\begin{equation} \label{eq:lex1}
A_{\circ}=A_{0} \edge{-\gamma_{1}} A_{1} 
\edge{-\gamma_{2}} \cdots \edge{-\gamma_{\ell}} A_{\ell}=A_{\lamm}, 
\end{equation}
and let $m_{\lamm}=\pi r_{i_1}r_{i_2} \cdots r_{i_{\ell}}$ be 
the corresponding reduced expression for $m_{\lamm}$ under \eqref{eq:1to1}. 
We construct $\CA(-\lamm)$ from this reduced expression 
$m_{\lamm}=\pi r_{i_1}r_{i_2} \cdots r_{i_{\ell}}$, 
which we denote by $\CA(-\lamm)_{\lex}$;
recall from Remark~\ref{rem:betaLPOS} and \eqref{eq:bk} that 
$\gamma_{k}=\gamma_{k}^{\OS}=
 \gamma_{k}^{\Le}=(\ol{\beta_{k}^{\Le}})^{\vee} \in -\lng(\Phi^{+} \setminus \Phi^{+}_{J})$ 
 for all $1 \le k \le \ell$. We set (see Remark~\ref{rem:betaLPOS})
%
%
\begin{equation} \label{eq:dk}
d_{k}:=\frac{b_{k}}{ \bpair{ \ol{\beta_{k}^{\Le}} }{-\lamm} } 
  = 1 - \frac{a_{k}}{ \bpair{ \ol{\beta_{k}^{\OS}} }{\lamm} }
\qquad \text{for $1 \le k \le \ell$}. 
\end{equation}
Because $m_{\lamm}=\pi r_{i_1}r_{i_2} \cdots r_{i_{\ell}}$ is the reduced 
expression corresponding to the lex $(-\lamm)$-chain of roots, 
it follows from \eqref{eq:bkle} that
\begin{equation*}
0 \le d_{1} \le d_{2} \le \cdots \le d_{\ell} < 1. 
\end{equation*}
%
%
\begin{rem} \label{rem:prec}
Let $1 \le k < p \le \ell$ be such that $d_{k}=d_{p}$. 
Then we know from \cite[Remark~6.5]{LNSSS2} that 
$\gamma_{k} \prec \gamma_{p}$ 
in the reflection order $\prec$ (see \eqref{eq:prec}). 
\end{rem}

In the following, 
we define a map $\Xi:\QBM_{\lex} \rightarrow \QLS(\lambda)$ 
(resp., $\Pi:\CA(-\lamm)_{\lex} \rightarrow \QLS(\lambda)$; see \cite[\S6.1]{LNSSS2}).
Let $A=\bigl\{j_{1} < \cdots < j_{r}\bigr\} \subset 
\bigl\{1,\,\dots,\,\ell\bigr\}$ be such that 
\begin{equation*}
p_{A}=\Bigl(
m_{\lamm} = z_{0} \edge{\beta_{j_1}^{\OS}} 
z_{1} \edge{\beta_{j_2}^{\OS}} \cdots \edge{\beta_{j_r}^{\OS}} z_{r}
\Bigr)
\end{equation*}
is an element of $\QBM_{\lex}$ (resp., $A \in \CA(-\lamm)_{\lex}$); 
if we set $x_{k}:=r_{\gamma_{j_1}} \cdots r_{\gamma_{j_k}} = \dir(z_{k}) \in W$ 
for $0 \le k \le r$, then 
\begin{equation*}
e=x_{0} \edge{\gamma_{j_1}} x_{1} \edge{\gamma_{j_2}} \cdots 
\edge{\gamma_{j_r}} x_{r}
\end{equation*}
is a directed path in the quantum Bruhat graph $\QB(W)$. 
Take $0=u_{0} \le u_{1} < u_{2} < \cdots < u_{s-1} < u_{s}=r$ (with $s \ge 1$) 
in such a way that
%
%
\begin{equation} \label{eq:ap}
\begin{split}
& \underbrace{0 = d_{j_1} = \cdots =d_{j_{u_1}}}_{=:\sigma_{0}} < 
  \underbrace{d_{j_{u_1+1}}= \cdots = d_{j_{u_2}}}_{=:\sigma_{1}} < \\[1.5mm]
& \hspace{20mm}
  \underbrace{d_{j_{u_2+1}}= \cdots = d_{j_{u_3}}}_{=:\sigma_{2}} < 
  \cdots \cdots \cdots 
  < \underbrace{d_{j_{u_{s-1}+1}}= \cdots =d_{j_{r}}}_{=:\sigma_{s-1}} < 1=:\sigma_{s};
\end{split}
\end{equation}
note that $u_{1} = 0$ if $d_{j_1} > 0$. 
We set $w_{p}':=x_{u_p}$ for $1 \le p \le s-1$, and 
$w_{s}':=x_{r}$. For each $1 \le p \le s-1$, we have 
a directed path 
\begin{equation*}
w_{p}'=x_{u_p} \edge{\gamma_{j_{u_p+1}}} 
x_{u_p+1} \edge{\gamma_{j_{u_p+2}}} \cdots 
\edge{\gamma_{j_{u_{p+1}}}} x_{u_{p+1}}=w_{p+1}'
\end{equation*}
in the quantum Bruhat graph $\QB(W)$. 
We claim that this directed path is a shortest directed path 
from $w_{p}'$ to $w_{p+1}'$. Indeed, 
since $d_{j_{u_p+1}}= \cdots = d_{j_{u_{p+1}}}$ by \eqref{eq:ap}, 
it follows from Remark~\ref{rem:prec} that 
$\gamma_{j_{u_p+1}} \prec \cdots \prec \gamma_{j_{u_{p+1}}}$ 
in the reflection order $\prec$ (see \eqref{eq:prec}). 
Therefore, we deduce from \cite[Theorem~6.4]{LNSSS2} that 
the directed path above is a shortest directed path 
from $w_{p}'$ to $w_{p+1}'$, as desired. Hence it follows that 
%
%
\begin{equation} \label{eq:achain}
\begin{split}
& w_{p}:=w_{p}'\lng=x_{u_p}\lng 
  \edger{-\lng\gamma_{j_{u_p+1}}} x_{u_p+1}\lng 
  \edger{-\lng\gamma_{j_{u_p+2}}} \\[1.5mm]
& \hspace{50mm} \cdots \cdots 
  \edger{-\lng\gamma_{j_{u_{p+1}}}} 
  x_{u_{p+1}}\lng=w_{p+1}'\lng=:w_{p+1}
\end{split}
\end{equation}
is also a shortest directed path in the quantum Bruhat graph $\QB(W)$, 
where $-\lng\gamma_{j_{u}} \in \Phi^{+} \setminus \Phi^{+}_{J}$ 
for all $u_{p}+1 \le u \le u_{p+1}$ 
since $\gamma_{j_{u}} \in -\lng\bigl(\Phi^{+} \setminus \Phi^{+}_{J}\bigr)$ 
as mentioned at the beginning of this subsection.
Moreover, for $u_p+1 \le u \le u_{p+1}$, we have 
\begin{equation*}
\sigma_{p}\pair{ -\lng\gamma_{j_{u}}^{\vee} }{\lambda} = 
d_{j_{u}}\pair{ \gamma_{j_{u}}^{\vee} }{-\lamm} = 
\frac{b_{j_u}}{ \bpair{ \ol{\beta_{j_u}^{\Le}} }{-\lamm} } 
\times \bpair{ \ol{ \beta_{j_u}^{\Le} } }{-\lamm} = b_{j_u} \in \BZ.
\end{equation*}
Hence the directed path \eqref{eq:achain} is 
a directed path in $\QB_{\sigma_p\lambda}(W)$. 
We deduce from \cite[Lemma~6.1]{LNSSS1} that 
there exists a directed path from $\mcr{w_{p+1}} = \mcr{w_{p+1}}^{J}$ to 
$\mcr{w_{p}} = \mcr{w_{p}}^{J}$ in $\QB_{\sigma_{p}\lambda}(W^{J})$. 
Therefore, we conclude that 
\begin{equation*}
\eta:=
(\mcr{w_{1}},\,\mcr{w_{2}},\,\dots,\,\mcr{w_{s}} \,;\, 
 \sigma_{0},\,\sigma_{1},\,\dots,\,\sigma_{s}) 
\in \QLS(\lambda);
\end{equation*}
we set $\Xi(p_{A}):=\eta$. 

\begin{rem} \label{rem:iota}
Keep the setting above. Because
$0 = d_{j_1} = \cdots =d_{j_{u_1}} < d_{j_{u_1 + 1}}$ by the definition of $u_{1}$, 
we see from \eqref{eq:bkle2} that $j_{u_{1}} \le M$ and $j_{u_1 +1} \ge M+1$, 
where $\ell(\mcr{\lng})=M$. Therefore, by definition \eqref{eq:tii}, 
$\ti{\iota}(p_{A})$ is just $\dir (z_{u_1}) = x_{u_{1}} = w_{1}'$. 
Hence we obtain 
%
%
\begin{equation} \label{eq:iota}
\iota(\Xi(p_{A})) = \iota(\eta) = \mcr{w_{1}} = \mcr{w_{1}'\lng} = 
\mcr{\ti{\iota}(p_{A})\lng}.
\end{equation}
\end{rem}
%
%
%
\begin{prop}[{\cite[Proposition~6.8 and Theorem~7.3]{LNSSS2}}] \label{prop:Pi}
The map $\Pi:\CA(-\lamm)_{\lex} \rightarrow \QLS(\lambda)$ is bijective. 
Moreover, for every $A \in \CA(-\lamm)_{\lex}$, 
\begin{equation} \label{eq:Pi-wtht}
\wt(\Pi(A)) = - \wt(A) \qquad \text{\rm and} \qquad 
\Deg (\Pi(A)) = - \Ht(A). 
\end{equation}
\end{prop}
%
%
%
\begin{prop} \label{prop:Xi}
The map $\Xi:\QBM_{\lex} \rightarrow \QLS(\lambda)$ is bijective. 
Moreover, for every $p_{A} \in \QBM_{\lex}$, 
\begin{equation*}
\wt(\Xi(p_{A}))= \wt(p_{A}) \qquad \text{\rm and} \qquad 
\Deg (\Xi(p_{A})) = - \deg(\qwt(p_{A})). 
\end{equation*}
\end{prop}

\begin{proof}
From the constructions, we see that 
the map $\Xi:\QBM_{\lex} \rightarrow \QLS(\lambda)$ above
is identical to the composite of the bijection 
$\QBM_{\lex} \stackrel{\sim}{\rightarrow} \CA(-\lamm)_{\lex}$ of Lemma~\ref{lem:A-QB} and 
the bijection $\Pi:\CA(-\lamm)_{\lex} \stackrel{\sim}{\rightarrow} \QLS(\lambda)$ 
in Proposition~\ref{prop:Pi}. Hence the map 
$\Xi:\QBM_{\lex} \rightarrow \QLS(\lambda)$ is also bijective.
\begin{equation*}
\begin{diagram}
\node{\quad \QBM_{\lex} \quad} 
\arrow{e,b}{\begin{subarray}{c}
  \text{Bijection} \\
  \text{in Lemma~\ref{lem:A-QB}}
  \end{subarray}
  }
\arrow{se,b}{\Xi}
\node{\CA(-\lamm)_{\lex}} 
\arrow{s,r}{\Pi} \\
\node{}
\node{\QLS(\lambda)}
\end{diagram}
\end{equation*}

We know from \eqref{eq:HtWt} that 
$\wt (A) = - \wt (p_{A})$ and $\Ht (A) = \deg (\qwt(p_{A}))$ 
for all $A \in \CA(-\lamm)$. 
Combining this equality and \eqref{eq:Pi-wtht}, 
we obtain the equalities 
$\wt(\Xi(p_{A}))= \wt(p_{A})$ and 
$\Deg(\Xi(p_{A}))= - \degr(\qwt(p_{A}))$
for all $p_{A} \in \QBM_{\lex}$, as desired.
\end{proof}
%
%
\begin{lem} \label{lem:Xi}
The image of $\QBM_{\lex,w}$ {\rm(}see Proposition~\ref{prop:embed}{\rm)}
under the bijection $\Xi:\QBM_{\lex} \rightarrow \QLS(\lambda)$ 
of Proposition~\ref{prop:Xi} is identical to $\QLS_{w}(\lambda)$. 
\end{lem}

\begin{proof}
Let $p \in \QBM_{\lex}$. Then,
\begin{align*}
p \in \QBM_{\lex,w} 
  & \stackrel{\eqref{eq:QBMlexw}}{\iff} \ti{\iota}(p) \ge w \mcr{\lng}^{-1} 
    \iff \ti{\iota}(p) \lng \le w \mcr{\lng}^{-1} \lng.
\end{align*}
Since $\ti{\iota}(p) \in W^{\omega(J)}$ (see Remark~\ref{rem:dirz}), 
it follows by \eqref{eq:mcrs} that
\begin{equation*}
\ti{\iota}(p) \lng \lngJ (\Phi^{+}_{J}) =
\ti{\iota}(p) \lng (- \Phi^{+}_{J}) =
\ti{\iota}(p) (\Phi^{+}_{\omega(J)}) \subset \Phi^{+}.
\end{equation*}
From this, we deduce that 
$\ti{\iota}(p) \lng \lngJ \in W^{J}$ again by \eqref{eq:mcrs}, 
which implies that 
$\mcr{\ti{\iota}(p) \lng} \lngJ = 
\mcr{\ti{\iota}(p) \lng \lngJ} \lngJ = 
(\ti{\iota}(p) \lng \lngJ) \lngJ = 
\ti{\iota}(p) \lng$. 
Therefore, 
\begin{equation*}
\ti{\iota}(p) \lng \le w \mcr{\lng}^{-1} \lng 
\iff \mcr{\ti{\iota}(p) \lng} \lngJ \le w \lngJ.
\end{equation*}
Here we have
\begin{equation*}
\mcr{\ti{\iota}(p) \lng} \lngJ \le w \lngJ \quad \iff \quad 
\mcr{\ti{\iota}(p) \lng} \le w.
\end{equation*}
Indeed, the ``only if'' part ($\Rightarrow$) follows immediately from 
\cite[Proposition~2.5.1]{BB}. Let us show the ``if'' part ($\Leftarrow$). 
Fix reduced expressions for $\lngJ \in W_{J}$ and $w \in W^{J}$, respectively,
and then take a reduced expression of $\mcr{\ti{\iota}(p) \lng} \in W^{J}$ 
that is a ``subword'' of the fixed reduced expression of $w$ 
(see \cite[Theorem~2.2.2]{BB}). By \cite[Proposition~2.4.4]{BB}, 
the concatenation of this reduced expression for $\mcr{\ti{\iota}(p) \lng}$ 
(resp., $w \in W^{J}$) with a reduced expression for $\lngJ$ is 
a reduced expression for $\mcr{\ti{\iota}(p) \lng} \lngJ$ 
(resp., $w \lngJ$); observe that the obtained reduced expression for
$\mcr{\ti{\iota}(p) \lng} \lngJ$ is a subword of 
the obtained reduced expression for $w\lngJ$. 
Therefore, by \cite[Theorem~2.2.2]{BB}, we see that 
$\mcr{\ti{\iota}(p) \lng} \lngJ \le w \lngJ$, as desired. 
Finally, we have 
\begin{align*}
\mcr{\ti{\iota}(p) \lng} \le w
 & \iff \iota(\Xi(p)) \le w \qquad \text{by \eqref{eq:iota}} \\
 & \iff \Xi(p) \in \QLS_{w}(\lambda). 
\end{align*}
This proves the lemma. 
\end{proof}

\begin{proof}[Proof of Theorem~\ref{thm:Mac0}]
We compute: 
\begin{align*}
\sum_{\eta \in \QLS_{w}(\lambda)} 
 e^{\wt(\eta)}q^{-\Deg(\eta)}
& = \sum_{p \in \QBM_{\lex,w}} 
 e^{\wt(p)}q^{\degr(\qwt(p))} \\
& \hspace{30mm} 
  \text{by Lemma~\ref{lem:Xi} and Proposition~\ref{prop:Xi}} \\[3mm]
& =\Mac{w\lambda} \quad \text{by \eqref{eq:prf1}}.
\end{align*}
This completes the proof of Theorem~\ref{thm:Mac0}. 
\end{proof}
%
%
\subsection{The formula in terms of the quantum alcove model.}
\label{subsec:qam}

We start with some review from \cite{LNSSS2}. 
Recall the Dynkin diagram automorphism $\omega:I \rightarrow I$ 
induced by $\lng \alpha_{j}=-\alpha_{\omega(j)}$ for $j \in I$. 
Note that $\omega$ acts as $-\lng$ 
on the integral weight lattice $X$. There exists a group automorphism, 
denoted also by $\omega$, of the Weyl group $W$ such that 
$\omega(r_{j})=r_{\omega(j)}$ for all $j \in I$.

Now, fix $\lambda \in X$ be a dominant integral weight 
with $J=\bigl\{ i \in I \mid \pair{\alpha_{i}^{\vee}}{\lambda} = 0 \bigr\}$, 
and let
%
%
\begin{equation} \label{eq:se1}
\eta=(x_{1},\,\dots,\,x_{s}\,;\,
\sigma_{0},\,\sigma_{1},\,\dots,\,\sigma_{s}) 
\in \QLS(\lambda),
\end{equation}
with $x_{1},\,\dots,\,x_{s} \in W^{J}$ and 
rational numbers $0=\sigma_{0} < \cdots < \sigma_{s}=1$. 
Then we define
%
%
\begin{equation} \label{eq:dual}
\eta^{\ast}:=
(\mcr{x_{s}w_\circ}^{\omega(J)},\,\dots,\,\mcr{x_{1}w_\circ}^{\omega(J)}\,;\,
 1-\sigma_{s},\,1-\sigma_{s-1},\,\dots,\,1-\sigma_{0})\,.
\end{equation}
We also define $\omega(\eta)$ by
%
%
\begin{equation} \label{eq:se2}
\omega(\eta):=
(\omega(x_{1}),\,\dots,\,\omega(x_{s})\,;\,
\sigma_{0},\,\sigma_{1},\,\dots,\,\sigma_{s}). 
\end{equation}
Both maps, $*$ and $\omega$, are 
bijections between $\QLS(\lambda)$ and $\QLS(-w_{\circ}\lambda)$, 
and they change the weight of a path by a negative sign and $\omega$, respectively. 
Finally, we set $S(\eta):=\omega(\eta^{\ast})=(\omega(\eta))^{\ast}$, 
which turns out to be the Lusztig involution on $\QLS(\lambda)$. 

Replacing $\lambda$ by $-\lamm$ 
in \S\ref{subsec:QB-A} and \S\ref{subsec:lex}, 
let us consider a lex $\lambda$-chain of roots, and 
the quantum alcove model $\CA(\lambda)_{\lex}$ associated to it.
Recall the map $\Pi$ (in Proposition~\ref{prop:Pi} 
with $\lambda$ replaced by $-\lamm$) and 
the corresponding commutative diagram:
\begin{equation}
\begin{diagram}
\node{\CA(\lambda)_{\lex} } \arrow{e,t}{\Pi} \arrow{se,b}{\Pi^{\ast}} 
\node{\QLS(-w_{\circ}\lambda)} \arrow{s,r}{\ast} \\
\node{} \node{\QLS(\lambda).}
\end{diagram}
\end{equation}

We need an analogue of \cite[Theorem~7.3]{LNSSS2} for the coheight statistic, 
which was defined in~\eqref{defcoheight}. This is stated as follows, 
and is proved in a completely similar way, based on the results in \cite{LNSSS2}.
%
%
\begin{thm} \label{energy-transl} 
Consider an admissible subset $A \in \CA(\lambda)_{\lex}$, and 
the corresponding QLS path $\Pi(A) \in \QLS(-\lamm)$. 
Write $\Pi(A)$ as follows {\rm(}cf. Definition~\ref{dfn:QLS}\rm{)}:
\begin{equation} \label{qlsp}
x_1 
\blarrl{-\sigma_1w_{\circ}\lambda} x_2
\blarrl{-\sigma_2w_{\circ}\lambda} \ldots 
\blarrl{-\sigma_{s-1}w_{\circ}\lambda} x_s,
\end{equation}
with $x_i \in W^{\omega(J)}$ and 
$0=\sigma_{0} < \sigma_{1} < \cdots < \sigma_{s} = 1$. Then, we have
\begin{equation}
 \cHt(A)=\sum_{i=1}^{s-1} \sigma_i\wt_{-w_\circ\lambda}(x_{i+1}\Rightarrow x_{i}), 
\end{equation}
where $\wt_{-w_\circ\lambda}(x_{i+1}\Rightarrow x_{i})$ 
was defined in \eqref{defwtlam}. 
\end{thm}

We will now express the nonsymmetric Macdonald polynomial 
in terms of the quantum alcove model. 
Recall that the final direction $\phi(A)$ of 
an admissible subset $A$ was defined in \eqref{deffinal}.
%
%
\begin{thm}\label{thmalc}
We have
\begin{equation} \label{e3}	
 E_{w\lambda}(x;q,0)=
  \sum_{ 
    \begin{subarray}{c}
     A \in \CA(\lambda) \\ 
     \mcr{\phi(A)}^{J} \le w
    \end{subarray}
  } q^{\cHt(A)} x^{\wt(A)}.
\end{equation}
\end{thm}

\begin{proof}
We derive this formula directly from Theorem~\ref{thm:Mac0}, based on the map $\Pi^\ast$, 
which is known to be a weight-preserving bijection, by \cite[Proposition~6.7]{LNSSS2}. 
Using the very explicit description of the map $\Pi^\ast$ in \cite[\S6.1]{LNSSS2}, 
we can see that it switches initial and final directions, i.e., 
for $A \in \CA(\lambda)$ we have
\begin{equation*}
\iota(\Pi^\ast(A))= \mcr{\phi(A)}^{J}.
\end{equation*}
Finally, by using the notation \eqref{qlsp} for $\Pi(A)$, 
we deduce:
\begin{align*}
\cHt(A)
 & = \sum_{u=1}^{s-1} \sigma_{u} \wt_{-\lamm}(x_{u+1}\Rightarrow x_{u}) 
   = - \Deg(S(\Pi(A))) \\[1.5mm]
 & = - \Deg(\omega(\Pi^\ast(A))) 
   = - \Deg(\Pi^\ast(A)).
\end{align*}
Here the first equality is based on Theorem~\ref{energy-transl}, 
the second one on \cite[Corollary~4.7]{LNSSS2}, 
the third one on the above definition of the Lusztig involution $S$, and 
the last one on \cite[Corollary~7.4]{LNSSS2}. 
\end{proof}

\begin{rem}
In \cite{LNSSS2}, we realized an appropriate tensor product of 
Kirillov--Reshetikhin crystals $\BB$ in terms of $\QLS(\lambda)$. 
Based on this, we expressed the so-called ``right'' energy function on $\BB$ as 
$\Deg(\eta)$ for $\eta\in\QLS(\lambda)$. In these terms, 
$\Deg(S(\eta))$ expresses the corresponding ``left'' energy function, 
see \cite[Remark~4.9]{LNSSS2}. We also realized  $\BB$ in terms of 
the quantum alcove model, and in this setup the two energy functions are 
expressed by the height and coheight statistics.
\end{rem}

When $\Gamma$ is an (arbitrary) $\lambda$-chain of roots, 
we denote by $\CA(\lambda)_{\Gamma}$ the quantum alcove model 
associated to $\Gamma$.
In \cite{LL2}, we defined certain combinatorial moves 
(called quantum Yang-Baxter moves) in the quantum alcove model, 
namely on the collection of $\CA(\lambda)_{\Gamma}$, 
where $\Gamma$ is any $\lambda$-chain (of roots).
We showed that these define an affine crystal isomorphism 
between $\CA(\lambda)_{\Gamma}$ and $\CA(\lambda)_{\Gamma'}$ 
for any two $\lambda$-chains $\Gamma$ and $\Gamma'$. 
We also showed that the moves preserve the weight, 
the height and coheight, as well as the final direction of 
(the path in $\QB(W)$ associated with) an admissible subset. 
Based on these facts, we can generalize Theorem~\ref{thmalc}. 

\begin{thm} \label{thmalcgen} 
Theorem~\ref{thmalc} still holds if we replace the admissible subsets 
$\CA(\lambda)_{\lex}$ for a lex $\lambda$-chain with 
the ones for an arbitrary $\lambda$-chain $\Gamma$, 
namely $\CA(\lambda)_{\Gamma}$.
\end{thm}

\begin{rem}
The formulas in Theorems \ref{thm:Mac0} and \ref{thmalc} 
(in fact, the latter can be replaced with the mentioned generalization) specialize, 
upon setting $q=0$, to the formulas for Demazure characters in terms of LS paths 
\cite[Theorem~5.2]{L-Inv} and the alcove model \cite[Theorem~6.3]{Le}.
\end{rem}
%
%
\section{Graded characters of quotients of Demazure modules.}
\label{sec:gc}
%
%
\subsection{Additional setting.}
\label{subsec:notation}

The untwisted affine Lie algebra $\Fg_{\af}$ is written as: 
$\Fg_{\af}=\Fg \otimes \BC[t,\,t^{-1}] \oplus \BC c \oplus \BC D$, 
where $c=\sum_{j \in I_{\af}} a^{\vee}_{j} \alpha_{j}^{\vee}$ 
is the canonical central element, and 
$D$ is the scaling element (or the degree operator); note that 
the Cartan subalgebra $\Fh_{\af}$ of $\Fg_{\af}$ is 
$\Fh \oplus \BC c \oplus \BC D$, 
where $\Fh$ is the Cartan subalgebra of $\Fg$. 
%

Let us denote by $\bigl\{\alpha_{i}\bigr\}_{i \in I_{\af}}$ and 
$\bigl\{\alpha_{i}^{\vee}\bigr\}_{i \in I_{\af}}$ 
the simple roots and simple coroots of $\Fg_{\af}$, respectively, and 
by $\Lambda_{j} \in \Fh_{\af}^{\ast}$, $j \in I_{\af}$, 
the fundamental weights for $\Fg_{\af}$; 
note that $\pair{D}{\alpha_{j}}=\delta_{j,0}$ and 
$\pair{D}{\Lambda_{j}}=0$ for $j \in I_{\af}$. 
We take a weight lattice $X_{\af}$ for $\Fg_{\af}$ as follows:
%
%
\begin{equation} \label{eq:lattices}
X_{\af} = 
\left(\bigoplus_{j \in I_{\af}} \BZ \Lambda_{j}\right) \oplus 
   \BZ \delta \subset \Fh_{\af}^{\ast}, 
\end{equation}
where $\delta \in \Fh_{\af}^{\ast}$ denotes the null root of $\Fg_{\af}$.
We think of a weight $\mu \in \Fh^{\ast}$ for $\Fg$ 
as a weight ($\in \Fh_{\af}^{\ast}$) for $\Fg_{\af}$ by: 
$\pair{c}{\mu}=\pair{D}{\mu}=0$. Then, for each $i \in I$, 
the fundamental weight $\vpi_{i}$ for $\Fg$ is identical to 
$\Lambda_{i}-a_{i}^{\vee}\Lambda_{0} \in \Fh_{\af}^{\ast}$; 
we call the weights $\vpi_{i}=\Lambda_{i}-a_{i}^{\vee}\Lambda_{0} 
\in \Fh_{\af}^{\ast}$, $i \in I$, the level-zero fundamental weights.

The (affine) Weyl group $W_{\af}$ of $\Fg_{\af}$ is the subgroup
$\langle r_{j} \mid j \in I_{\af} \rangle \subset \GL(\Fh_{\af}^{\ast})$ 
generated by the simple reflections $r_{j}$
associated to $\alpha_{j}$ for $j \in I_{\af}$, 
with length function $\ell:W_{\af} \rightarrow \BZ_{\ge 0}$ and 
unit element $e \in W_{\af}$; recall that $W_{\af} \cong W \ltimes Q^{\vee}$. 
We denote by $\rr$ the set of real roots, 
and by $\prr \subset \rr$ the set of positive real roots. 
%
%
%
\begin{dfn}[{\cite{Pet97}}] \label{dfn:sell}
Let $x \in W_{\af} \cong W \ltimes Q^{\vee}$, and 
write it as $x = w t_{\xi}$ for $w \in W$ and $\xi \in Q^{\vee}$. 
Then we define the semi-infinite length $\sell(x)$ of $x$ by
$\sell (x) := \ell (w) + 2 \pair{\xi}{\rho}$. 
\end{dfn}

Now, let $J$ be a subset of $I$. Following \cite{Pet97} 
(see also \cite[\S10]{LS10}), we define
\begin{align}
(\Phi_J)_{\af}^{+}
  &:= \biggl(\bigoplus_{i \in J} \BZ \alpha_{i} + \BZ\delta \biggr) \cap \prr, \\
\label{eq:stabilizer}
(W^J)_{\af}
 &:= \bigl\{ x \in W_{\af} \mid 
 \text{$x\beta \in \prr$ for all $\beta \in (\Phi_J)_{\af}^+$} \bigr\}. 
\end{align}
%
%
\begin{dfn}\label{def:SiB}
(1) The (parabolic) semi-infinite Bruhat graph $\SB$ is 
the $\prr$-labeled, directed graph with vertex set $(W^J)_{\af}$ 
and $\prr$-labeled, directed edges of the following form:
$x \edge{\beta} r_{\beta} x$ for $x \in (W^J)_{\af}$ and $\beta \in \prr$, 
where $r_{\beta } x \in (W^J)_{\af}$ and 
$\sell (r_{\beta} x) = \sell (x) + 1$.

(2)
The semi-infinite Bruhat order is a partial order 
$\sile$ on $(W^J)_{\af}$ defined as follows: 
for $x,\,y \in (W^J)_{\af}$, we write $x \sile y$ 
if there exists a directed path from $x$ to $y$ in $\SB$; 
also, we write $x \sil y$ if $x \sile y$ and $x \ne y$. 
\end{dfn}

Finally, let $U_{q}(\Fg_{\af})$ denote the quantum affine algebra associated to $\Fg_{\af}$ 
with integral weight lattice $X_{\af}$, and $E_{j},\,F_{j},\,j \in I_{\af}$, 
the Chevalley generators of $U_{q}(\Fg_{\af})$. 
Also, let $U_{q}^{+}(\Fg_{\af})$ denote
the subalgebra of $U_{q}(\Fg_{\af})$ generated by $E_{j}$, $j \in I_{\af}$.

%
%
\subsection{Extremal weight modules and Demazure modules.}
\label{subsec:ext}

For an arbitrary integral weight $\lambda \in X_{\af}$ of $\Fg_{\af}$, 
let $V(\lambda)$ denote the extremal weight module of 
extremal weight $\lambda$ over $U_{q}(\Fg_{\af})$, which is 
an integrable $U_{q}(\Fg_{\af})$-module generated 
by a single element $v_{\lambda}$ with 
the defining relation that $v_{\lambda}$ is 
an ``extremal weight vector'' of weight $\lambda$
(for details, see \cite[\S8]{K-mod} and \cite[\S3]{K-lv0}).
We know from \cite[Proposition~8.2.2]{K-mod} that $V(\lambda)$ has 
a crystal basis $(\CL(\lambda),\,\CB(\lambda))$ with corresponding global basis 
$\bigl\{G(b) \mid b \in \CB(\lambda)\bigr\}$; 
we denote by $u_{\lambda}$ the element of $\CB(\lambda)$ 
such that $G(u_{\lambda})=v_{\lambda} \in V(\lambda)$. 

Now, let $\lambda$ be a dominant integral weight for $\Fg$, and set 
$J=J_{\lambda}=\bigl\{i \in I \mid \pair{\alpha_{i}^{\vee}}{\lambda}=0\bigr\}$; 
note that $\lambda$ is regarded as an element of $X_{\af}$ by 
$\pair{c}{\lambda}=\pair{D}{\lambda}=0$. 
For each $x \in W_{\af}$, we set
%
%
\begin{equation} \label{eq:dem}
V_{x}^{+}(\lambda):=U_{q}^{+}(\Fg_{\af})S_{x}^{\norm}v_{\lambda} 
\subset V(\lambda), 
\end{equation}
where $S^{\norm}_{x}$ denotes the action of the (affine) Weyl group $W_{\af}$ on 
the set of extremal weight vectors (see \cite[(3.2.1)]{NS}). 
We know from \cite[\S2.8]{K-rims} (see also \cite[\S4.1]{NS}) that 
there exists a subset $\CB_{x}^{+}(\lambda)$ of 
the crystal basis $\CB(\lambda)$ such that 
$\bigl\{G(b) \mid b \in \CB_{x}^{+}(\lambda)\bigr\}$ is 
the global basis of $V_{x}^{+}(\lambda)$.

%
\subsection{Quotients of Demazure modules and their graded characters.}
\label{subsec:gch}

We fix a dominant integral weight $\lambda$ for $\Fg$. 
As in \cite[\S7.2]{NS}, we set 
\begin{equation*}
Z_{\lng}^{+}(\lambda) :=
\sum_{
 \begin{subarray}{c}
  \bc \in \ol{\Par(\lambda)} \\[1.5mm]
  \bc \ne (\emptyset)_{i \in I}
 \end{subarray}
 } U_{q}^{+}(\Fg_{\af}) S_{\bc}S_{\lng}^{\norm}v_{\lambda};
\end{equation*}
notice that our notation differs slightly from that in~\cite{NS}.
Here, $\ol{\Par(\lambda)}$ denotes a certain set of multi-partitions 
indexed by $I$ (see \cite[(2.5.1)]{NS}), and 
$S_{\bc} \in U_{q}^{+}(\Fg_{\af})$ denotes the PBW-type basis element 
of weight $|\bc|\delta$ corresponding to 
the ``purely imaginary part'' (see \cite[page 352]{BN}), 
where $|\bc|$ is the sum of all parts 
in the multi-partition $\bc$. 
Notice that $Z_{\lng}^{+}(\lambda) \subset V_{\lng}^{+}(\lambda)=
U_{q}^{+}(\Fg_{\af})S_{\lng}^{\norm}v_{\lambda}$ since 
$S_{\bc} \in U_{q}^{+}(\Fg_{\af})$ for all $\bc \in \ol{\Par(\lambda)}$. 

Now, let $w \in W$; in what follows, we may assume that 
$w \in W^{J} \subset (W^{J})_{\af}$ since 
$V_{w}^{+}(\lambda) = V_{\mcr{w}}^{+}(\lambda)$ 
for $w \in W$ by \cite[Lemma~4.1.2]{NS}. 
Then, noting that 
$V_{w}^{+}(\lambda)=V_{\mcr{w}}^{+}(\lambda) 
 \subset V_{\mcr{\lng}}^{+}(\lambda) = V_{\lng}^{+}(\lambda)$ 
by \cite[Corollary~5.2.5]{NS} since $\mcr{w} \sile \mcr{\lng}$, 
we define $U_{w}^{+}(\lambda)$ to be the image of 
$V_{w}^{+}(\lambda)$ under the canonical projection 
$V_{\lng}^{+}(\lambda) \twoheadrightarrow 
 V_{\lng}^{+}(\lambda)/Z_{\lng}^{+}(\lambda)$.
We write the weight space decomposition of $U_{w}^{+}(\lambda)$
with respect to $\Fh_{\af}$ as:
\begin{equation*}
U_{w}^{+}(\lambda) = \bigoplus_{\gamma \in Q,\,k \in \BZ} 
U_{w}^{+}(\lambda)_{\lambda-\gamma+k\delta}, 
\end{equation*}
and define the graded character
$\gch U_{w}^{+}(\lambda)$ of $U_{w}^{+}(\lambda)$ to be 
\begin{equation*}
\gch U_{w}^{+}(\lambda) := \sum_{\gamma \in Q,\,k \in \BZ} 
\dim U_{w}^{+}(\lambda)_{\lambda-\gamma+k\delta}\,x^{\lambda-\gamma}q^{k}, 
\qquad \text{where $q=x^{\delta}$}.
\end{equation*}
The following is the main result of this section.
%
%
\begin{thm} \label{thm:gch}
Keep the notation and setting above. We have
\begin{equation*}
\gch U_{w}^{+}(\lambda) = \Mac{w\lambda}.
\end{equation*}
\end{thm}
%
%
\subsection{Semi-infinite Lakshmibai-Seshadri paths.}
\label{subsec:SLS}

We keep the notation and setting of \S\ref{subsec:gch}; 
recall that $\lambda = \sum_{i \in I} m_{i} \vpi_{i}$ is 
a dominant integral weight for $\Fg$, and 
$J=J_{\lambda}= \bigl\{ i \in I \mid 
\pair{\alpha_i^{\vee}}{\lambda}=0 \bigr\} \subset I$.
%
%
\begin{dfn}\label{dfn:SBa}
For a rational number $0 < \tau < 1$, 
define $\SBb{\tau}$ to be the subgraph of $\SB$ 
with the same vertex set but having only the edges of the form:
$x \edge{\beta} y$ with 
$\tau \pair{\beta^{\vee}}{x\lambda} \in \BZ$.
\end{dfn}
%
%
\begin{dfn}\label{dfn:SiLS}
A semi-infinite Lakshmibai-Seshadri path (SiLS path for short) of 
shape $\lambda $ is, by definition, a pair $(\bm{y}\,;\,\bm{\tau})$ of 
a (strictly) decreasing sequence $\bm{y} : y_1 \sig \cdots \sig y_s$ 
of elements in $(W^J)_{\af}$ and an increasing sequence 
$\bm{\tau} : 0 = \tau_0 < \tau_1 < \cdots  < \tau_s =1$ 
of rational numbers satisfying the condition that 
there exists a directed path from $y_{u+1}$ to  $y_u$ in 
$\SBb{\tau_u}$ for each $u = 1,\,2,\,\dots,\,s-1$. 
We denote by $\sLS$ the set of all SiLS paths of shape $\lambda$.
\end{dfn}

In \cite[\S3.1]{INS}, we defined root operators 
$e_{j}$ and $f_{j}$, $j \in I_{\af}$, on $\sLS$, 
and proved that the set $\sLS$, equipped with these root operators, 
is a crystal with weights in $X_{\af}$. 
%
%
\begin{thm}[{\cite[Theorem~3.2.1]{INS}}] \label{thm:INS}
Keep the notation and setting above. 
There exists an isomorphism of crystals 
between the crystal basis $\CB(\lambda)$ of the extremal weight module 
$V(\lambda)$ of extremal weight $\lambda$ and the crystal $\sLS$ of 
SiLS paths of shape $\lambda$.
\end{thm}
%
%
\begin{rem} \label{rem:INS}
For $\pi=(y_1,\,\dots,\,y_s\,;\,\tau_{0},\,\tau_{1},\,\dots,\,\tau_{s}) \in \sLS$, 
we define the piecewise-linear, continuous map 
$\ol{\pi}: [0,1] \rightarrow \BR \otimes_{\BZ} X_{\af}$ by
%
%
\begin{equation} \label{eq:path2}
\ol{\pi}(t)=\sum_{p=1}^{u-1}
(\tau_{p}-\tau_{p-1})y_{p}\lambda+ 
(t-\tau_{u-1})y_{u}\lambda \quad 
\text{for $\tau_{u-1} \le t \le \tau_{u}$, $1 \le u \le s$}.
\end{equation}
Then we know from \cite[Proposition~3.1.3]{INS} that 
$\ol{\pi}$ is a Lakshmibai-Seshadri (LS for short) path of shape $\lambda$; 
for the definition of LS paths of shape $\lambda$, 
see \cite{L} and \cite[\S2.2 and 2.3]{LNSSS2}. 
We denote by $\BB(\lambda)$ the set of all LS paths of shape $\lambda$. 
In fact, the map $\ol{\phantom{\pi}} : 
\sLS \rightarrow \BB(\lambda)$, $\pi \mapsto \ol{\pi}$, is 
a surjective, strict crystal morphism. 
\end{rem}

Define a surjective map 
$\cl : (W^{J})_{\af} \twoheadrightarrow W^{J}$ by
\begin{equation*}
\cl (x) := w \quad \text{if $x = wzt_{\xi}$ 
   for $w \in W^{J}$, $z \in W_{J}$, and $\xi \in Q^{\vee}$.}
\end{equation*}
Then, for $\pi = (y_{1},\,\dots,\,y_{s}\,;\,\tau_{0},\,\tau_{1},\,\dots,\,\tau_{s}) 
\in \sLS$, we set (see \cite[Remark~6.2.1]{NS})
\begin{equation*}
\cl(\pi): = 
 (\cl(y_{1}),\,\dots,\,\cl(y_{s})\,;\,\tau_{0},\,\tau_{1},\,\dots,\,\tau_{s});
\end{equation*}
for each $1 \le p < q \le s$ such that $\cl(y_{p})= \cdots = \cl(y_{q})$, 
we drop $\cl(y_{p}),\,\dots,\,\cl(y_{q-1})$ and $\tau_{p},\,\dots,\,\tau_{q-1}$
from this expression of $\cl(\pi)$. 

Let $\BB_{0}^{\si}(\lambda)$ denote the connected component of $\sLS$ 
containing $\pi_{e}:=(e\,;\,0,\,1)$. We know from 
\cite[Lemma~7.1.2]{NS} that for each $\eta \in \QLS(\lambda)=\BB(\lambda)_{\cl}$ 
(see Remark~\ref{rem:QLS}), there exists a unique element 
$\pi_{\eta}=(y_1,\,\dots,\,y_s\,;\,\bm{\tau}) \in \BB_{0}^{\si}(\lambda)$ such that 
$\iota(\pi_{\eta}):=y_1 \in W^{J}$ and $\cl(\pi_{\eta})=\eta$. 
We claim that
%
%
\begin{equation} \label{eq:deg}
\wt(\pi_{\eta}) = \lambda-\beta - \Deg(\eta) \delta, \qquad 
\text{where $\beta \in Q^{+}:=\sum_{j \in I} \BZ_{\ge 0} \alpha_{j}$}. 
\end{equation}
Indeed, since $\wt(\pi_{\eta}) = \wt(\ol{\pi_{\eta}})$ by their definitions, 
it suffices to show that $\ol{\pi_{\eta}} \in \BB(\lambda)$ satisfies 
the following conditions (see \cite[Proposition~3.1.3]{NSdeg} and 
\cite[\S4.2 and Theorem~4.6]{LNSSS2}):
\begin{enu}
\item[(a)] $\cl(\ol{\pi_{\eta}}(t)) = \eta(t)$ for all $t \in [0,1]$, 
where $\cl:\BR \otimes_{\BZ} X_{\af} \twoheadrightarrow 
(\BR \otimes_{\BZ} X_{\af})/\BR\delta$ denotes the canonical projection; 

\item[(b)] $\ol{\pi_{\eta}}$ is contained 
in the connected component $\BB_{0}(\lambda)$ of 
$\BB(\lambda)$ containing $\pi_{\lambda}$, where 
$\pi_{\lambda}(t):=t\lambda$ for $t \in [0,1]$; 

\item[(c)] $\iota(\ol{\pi_{\eta}})=y_{1}\lambda \in \lambda-Q^{+}$. 
\end{enu}
If $x \in W_{\af}$ is of the form 
$x = wzt_{\xi}$ with $w \in W^{J}$, $z \in W_{J}$, and $\xi \in Q^{\vee}$, 
then $x\lambda = w\lambda - \pair{\xi}{\lambda}\delta$ 
(recall that $\pair{c}{\lambda}=0$), and hence 
$x\lambda \equiv w\lambda$ modulo $\BR\delta$. 
Therefore, assertion (a) is obvious from the definitions of 
$\ol{\phantom{\pi}} : \sLS \rightarrow \BB(\lambda)$ and the maps $\cl$.
Also, since $\pi_{\eta} \in \BB^{\si}_{0}(\lambda)$, 
there exists a monomial $Y$ 
in root operators such that $\pi_{\eta}=Y\pi_{e}$. 
Because $\ol{\phantom{\pi}} : \sLS \rightarrow \BB(\lambda)$ 
commutes with the action of root operators, 
we have $\ol{\pi_{\eta}}=\ol{Y\pi_{e}}=Y\ol{\pi_{e}}=Y\pi_{\lambda}$. 
Hence we obtain $\ol{\pi_{\eta}} \in \BB_{0}(\lambda)$.
Finally, since $\iota(\pi_{\eta})=y_{1} \in W^{J}$ and 
$\lambda$ is a dominant integral weight for $\Fg$, 
it follows that $\iota(\ol{\pi_{\eta}}) = y_{1}\lambda$ 
is contained in $\lambda-Q^{+}$. This proves \eqref{eq:deg}. 
%
%
\subsection{Proof of Theorem~\ref{thm:gch}.}
\label{subsec:prf-gch}

We know from \cite[Theorem~7.2.2\,(1)]{NS} that 
there exists a subset $\CB(Z_{\lng}^{+}(\lambda))$ of $\CB(\lambda)$ 
such that $\bigl\{G(b) \mid b \in \CB(Z_{\lng}^{+}(\lambda))\bigr\}$ is 
the global basis of $Z_{\lng}^{+}(\lambda)$. Also, recall that 
$\bigl\{G(b) \mid b \in \CB_{w}^{+}(\lambda)\bigr\}$ is 
the global basis of $V_{w}^{+}(\lambda)$. Therefore,
\begin{equation*}
\bigl\{G(b) \ \mathrm{mod} \ Z_{\lng}^{+}(\lambda) \mid 
b \in \CB(U_{w}^{+}(\lambda)):=
\CB_{w}^{+}(\lambda) \setminus \CB(Z_{\lng}^{+}(\lambda)) \bigr\}
\end{equation*}
is the global basis of $U_{w}^{+}(\lambda)$, 
which is the image of $V_{w}^{+}(\lambda)$ under
the canonical projection 
$V_{\lng}^{+}(\lambda) \twoheadrightarrow 
 V_{\lng}^{+}(\lambda)/Z_{\lng}^{+}(\lambda)$.

We know from \cite[Theorem~7.2.2\,(2)]{NS} that
there exists an isomorphism $\Psi_{\lambda}^{\vee}:
\CB(\lambda) \stackrel{\sim}{\rightarrow} \BB^{\si}(\lambda)$ 
of crystals, which maps $\CB(U_{w}^{+}(\lambda)) \subset \CB(\lambda)$ to
\begin{equation*}
\bigl\{\pi_{\eta} \mid 
 \text{\rm $\eta \in \QLS(\lambda)$ such that $w \ge \iota(\pi_{\eta})$} \bigr\} 
 \subset \sLS.
\end{equation*}
Since $\iota(\pi_{\eta}) \in W^{J}$, we see that $\iota(\pi_{\eta}) = \iota(\eta)$. 
Therefore, the subset above is identical to the set 
$\bigl\{\pi_{\eta} \mid \eta \in \QLS_{w}(\lambda) \bigr\}$. 
Hence we compute:
\begin{align*}
\gch U_{w}^{+}(\lambda) 
 & = \left(\sum_{b \in \CB(U_{w}^{+}(\lambda))} x^{\wt (b)}\right)\Biggm|_{x^{\delta}=q}
   = \left(\sum_{\eta \in \QLS_{w}(\lambda)} x^{\wt(\pi_{\eta})}\right)\Biggm|_{x^{\delta}=q} \\[3mm]
 & = \left(\sum_{\eta \in \QLS_{w}(\lambda)} x^{\wt(\eta)-\Deg(\eta)\delta}\right)\Biggm|_{x^{\delta}=q}
   \qquad \text{by \eqref{eq:deg}} \\[3mm]
& = \sum_{\eta \in \QLS_{w}(\lambda)} q^{-\Deg(\eta)} x^{\wt(\eta)}
  = \Mac{w\lambda} \qquad \text{by Theorem~\ref{thm:Mac0}}.
\end{align*}
This completes the proof of Theorem~\ref{thm:gch}.

\appendix

\section*{Appendix.}

%
\section{Recursive formulas in terms of Demazure operators.}
\label{sec:rec}
We use the notation of \S\ref{subsec:setting} and \S\ref{subsec:notation}. 
Fix a dominant integral weight $\lambda \in X$, and 
set $J=J^{\lambda}=\bigl\{i \in I \mid \pair{\alpha_{i}^{\vee}}{\lambda}=0\bigr\}$. 
For each $i \in I$, we define a $\BZ[q]$-linear operator $D_{i}$ 
(called a Demazure operator) on $\bigl(\BZ[q]\bigr)[e^{\xi}\,;\,\xi \in X]$ by: 
\begin{align}
D_{i}(e^{\xi}) & := 
\frac{ e^{\xi+\rho} - e^{r_{i}(\xi+\rho)} }
     { 1-e^{-\alpha_{i}} } e^{-\rho} \nonumber \\[3mm]
& =
\begin{cases}
e^{\xi}\bigl(1+e^{-\alpha_{i}}+ \cdots +e^{-n\alpha_{i}}\bigr) 
  & \text{if $n=\pair{\alpha_{i}^{\vee}}{\xi} \ge 0$}, \\[1.5mm]
0 & \text{if $n=\pair{\alpha_{i}^{\vee}}{\xi}=-1$}, \\[1.5mm]
-e^{\xi}\bigl(e^{\alpha_{i}}+ \cdots +e^{(-n-1)\alpha_{i}}\bigr)
  & \text{if $n=\pair{\alpha_{i}^{\vee}}{\xi} \le -2$}. 
\end{cases} \label{eq:Demazure}
\end{align}
In this appendix, we give a recursive formula 
for $\gch \QLS_{w}(\lambda)$ (Proposition~\ref{prop:dem}) 
and one for $\Mac{w\lambda}$ (Proposition~\ref{prop:demM}), 
both of which are in terms of Demazure operators. 
%
%
\subsection{Recursive formula for $\gch \QLS_{w}(\lambda)$.}
\label{sec:rec-gch}
%
%
\begin{prop} \label{prop:dem}
Let $w \in W^{J}$ and $i \in I$ be such that $w > r_{i}w$; 
note that $r_{i}w \in W^{J}$ by {\rm \cite[Lemma~5.8]{LNSSS1}}. 
Then we have 
\begin{equation*}
\gch \QLS_{w}(\lambda) = D_{i} \gch \QLS_{r_{i}w}(\lambda). 
\end{equation*}
\end{prop}

Let $U_{q}'(\Fg_{\af})$ denote the quantum affine algebra 
without the degree operator associated to $\Fg_{\af}$. 
We know that the set $\QLS(\lambda)=\BB(\lambda)_{\cl}$ (see Remark~\ref{rem:QLS}), 
equipped with root operators $e_{j}$, $f_{j}$, $j \in I_{\af}$, 
is a $U_{q}'(\Fg_{\af})$-crystal; for the definition of root operators, 
see \cite[\S2.3]{LNSSS2} and \cite[\S2.2]{NSdeg}.
We prove Proposition~\ref{prop:dem} 
by using this $U_{q}'(\Fg_{\af})$-crystal structure on 
$\QLS(\lambda)=\BB(\lambda)_{\cl}$ 
(cf. \cite[Theorem in \S5.2]{L-Inv}). 
%
%
\begin{lem}[recursive relation] \label{lem:rec}
Let $w \in W^{J}$ and $i \in I$ be such that $w > r_{i}w$. 
We have 
\begin{equation*}
\QLS_{w}(\lambda) = 
 \bigcup_{p \ge 0} f_{i}^{p}\QLS_{r_{i}w}(\lambda) \setminus \{\bzero\}. 
\end{equation*}
\end{lem}

\begin{proof}
First we prove the inclusion $\subset$. 
Let $\eta \in \QLS_{w}(\lambda)$, and set $\eta':=e_{i}^{\max}\eta$.
It suffices to show that $\eta' \in \QLS_{r_{i}w}(\lambda)$; 
for simplicity of notation, we set $x:=\iota(\eta) \in W^{J}$. 
If $\iota(\eta')=\iota(\eta)=x$, then it follows from 
the definition of the root operator $e_{i}$ that 
\begin{equation*}
\pair{\alpha_{i}^{\vee}}{x\lambda} = 
\pair{\alpha_{i}^{\vee}}{\iota(\eta)\lambda} = 
\pair{\alpha_{i}^{\vee}}{\iota(\eta')\lambda} \ge 0, 
\end{equation*}
since $e_{i}\eta' = \bzero$. 
Because $\eta \in \QLS_{w}(\lambda)$ by the assumption, 
we have $x=\iota(\eta) \le w$. 
Also, from the assumption that $w > r_{i}w$ and $w \in W^{J}$, 
it follows that $r_{i}w \in W^{J}$ and $\pair{\alpha_{i}^{\vee}}{w\lambda} < 0$ 
by \cite[Lemmas~5.8 and 5.9]{LNSSS1}. 
Therefore, we deduce from \cite[Lemma~4.1\,a)]{L} 
applied to $x \le w$ that 
$x \le r_{i}w$, and hence $\iota(\eta') = \iota(\eta) = x \le r_{i}w$. 
Thus we obtain $\eta' \in \QLS_{r_{i}w}(\lambda)$, 
as desired. 
If $\iota(\eta') \ne \iota(\eta)$, then 
it follows from the definition of 
the root operator $e_{i}$ that
$\iota(\eta')=r_{i}\iota(\eta)=r_{i}x$ and 
\begin{equation*}
\pair{\alpha_{i}^{\vee}}{x\lambda} 
  = - \pair{\alpha_{i}^{\vee}}{r_{i}x\lambda} 
  = - \pair{\alpha_{i}^{\vee}}{r_{i}\iota(\eta)\lambda}
  = - \pair{\alpha_{i}^{\vee}}{r_{i}\iota(\eta')\lambda} < 0. 
\end{equation*}
Since $x=\iota(\eta) \le w$ by the assumption 
and $\pair{\alpha_{i}^{\vee}}{w\lambda} < 0$ as seen above, 
we deduce from \cite[Lemma~4.1\,c)]{L} applied to $x \le w$ that 
$r_{i}x \le r_{i}w$, and hence $\iota(\eta') = r_{i}x \le r_{i}w$. 
Therefore, we obtain $\eta' \in \QLS_{r_{i}w}(\lambda)$, 
as desired. This proves the inclusion $\subset$.

Next we prove the opposite inclusion $\supset$. 
Let $\eta' \in \QLS_{r_{i}w}(\lambda)$, and 
assume that $\eta:=f_{i}^{p}\eta' \ne \bzero$ 
for some $p \ge 0$. If $\iota(\eta)=\iota(\eta')$, 
then 
\begin{equation*}
\iota(\eta)=\iota(\eta') \le r_{i}w < w, 
\end{equation*}
and hence $\eta \in \QLS_{w}(\lambda)$. 
Assume now that $\iota(\eta) \ne r_{i}\iota(\eta')$, and hence 
$\iota(\eta)=r_{i}\iota(\eta')$; 
for simplicity of notation, 
we set $x':=\iota(\eta') \in W^{J}$. 
Then we see from the definition of the root operator $f_{i}$ that 
$\pair{\alpha_{i}^{\vee}}{x'\lambda} = \pair{\alpha_{i}^{\vee}}{\iota(\eta')\lambda} > 0$. 
Also, we have $\pair{\alpha_{i}^{\vee}}{r_{i}w\lambda} = 
- \pair{\alpha_{i}^{\vee}}{w\lambda} > 0$ as seen above and 
$x' = \iota(\eta') \le r_{i}w$ by the assumption. 
It follows from \cite[Lemma~4.1\,c)]{L} applied to $x' \le r_{i}w$
that $r_{i}x' \le w$, and hence $\iota(\eta) = r_{i}x' \le w$. 
Therefore, we obtain $\eta \in \QLS_{w}(\lambda)$, as desired. 
This proves the opposite inclusion $\supset$. 
This completes the proof of the lemma. 
\end{proof}

For $i \in I$ and $\eta \in \QLS(\lambda)$, 
let $S_{i}(\eta)$ denote the $\alpha_{i}$-string through $\eta$, that is, 
\begin{equation*}
S_{i}(\eta) : =
 \bigl\{ e_{i}^{p}\eta,\,f_{i}^{q}\eta \mid p,\,q \ge 0 \bigr\}
 \setminus \{\bzero\}.
\end{equation*}
%
%
\begin{lem}[string property] \label{lem:string}
Let $\eta \in \QLS(\lambda)$, and $i \in I$. 
For $z \in W^{J}$, 
\begin{equation*}
\QLS_{z}(\lambda) \cap S_{i}(\eta)=
\emptyset,\ 
\bigl\{e_{i}^{\max}\eta\bigr\},\ \text{\rm or}\ 
S_{i}(\eta).
\end{equation*}
\end{lem}

\begin{proof}
For simplicity of notation, we set $\eta':=e_{i}^{\max}\eta$. 
We will prove that if 
$\QLS_{z}(\lambda) \cap S_{i}(\eta)$ is neither $\emptyset$ nor 
$\bigl\{e_{i}^{\max}\eta\bigr\}$, then 
$\QLS_{z}(\lambda) \cap S_{i}(\eta) = S_{i}(\eta)$, 
or equivalently, $S_{i}(\eta) \subset \QLS_{z}(\lambda)$. 
By our assumption, 
$\QLS_{z}(\lambda) \cap S_{i}(\eta)$ contains 
an element $\eta''$ that is not $\eta'$. 
We can write the element $\eta''$ as $\eta''= f_{i}^{p}\eta'$ 
for some $p \ge 1$. 
Here, from the definition of the root operator $f_{i}$, 
we can deduce that 
\begin{equation*}
\iota(f_{i}\eta') = \iota(f_{i}^{2}\eta') = \cdots = 
\iota(f_{i}^{p}\eta') = \cdots = \iota(f_{i}^{\max}\eta').
\end{equation*}
Since $\iota(f_{i}^{p}\eta') = \iota(\eta'') \le z$ by the assumption that 
$\eta'' \in \QLS_{z}(\lambda)$, 
we see that the elements $f_{i}\eta',\,f_{i}^{2}\eta',\,\dots,\,
f_{i}^{\max}\eta'$ are all contained in $\QLS_{z}(\lambda)$. Namely, 
%
%
\begin{equation} \label{eq:lemdem1}
S_{i}(\eta) \setminus \bigl\{\eta'\bigr\} \subset 
\QLS_{z}(\lambda). 
\end{equation}
Hence it remains to show that $\eta' \in \QLS_{z}(\lambda)$. 
If $\iota(\eta')=\iota(\eta'')$, then we have 
$\iota(\eta') \le z$ since $\iota(\eta'') \le z$ by 
the assumption that $\eta'' \in \QLS_{z}(\lambda)$. 
This implies that $\eta' \in \QLS_{z}(\lambda)$. 
Assume now that $\iota(\eta'') \ne r_{i}\iota(\eta')$, and hence
that $\iota(\eta'')=r_{i}\iota(\eta')$. 
Then, by the definition of the root operator $f_{i}$, 
we see that $\pair{\alpha_{i}^{\vee}}{\iota(\eta')\lambda} > 0$. 
Therefore, we deduce that $\pair{\alpha_{i}^{\vee}}{\iota(\eta'')\lambda} = 
\pair{\alpha_{i}^{\vee}}{r_{i}\iota(\eta')\lambda} < 0$, and hence 
that $\iota(\eta')=r_{i}\iota(\eta'') < \iota(\eta'') \le z$.
Thus we obtain $\eta' \in \QLS_{z}(\lambda)$.
Combining this with \eqref{eq:lemdem1}, we conclude that 
$S_{i}(\eta) \subset \QLS_{z}(\lambda)$, as desired. 
This completes the proof of the lemma. 
\end{proof}

\begin{proof}[Proof of Proposition~\ref{prop:dem}]
First, we show that for each $\eta \in \QLS(\lambda)$, 
%
%
\begin{equation} \label{eq:string}
\QLS_{w}(\lambda) \cap S_{i}(\eta)=
\emptyset \text{ or } S_{i}(\eta). 
\end{equation}
Now, assume that 
$\QLS_{w}(\lambda) \cap S_{i}(\eta) \ne \emptyset$. 
Then, we see from Lemma~\ref{lem:string} that 
$\QLS_{w}(\lambda) \cap S_{i}(\eta) = 
\bigl\{e_{i}^{\max}\eta\bigr\}$ or $S_{i}(\eta)$; 
in both cases, we have $e_{i}^{\max}\eta \in \QLS_{w}(\lambda)$. 
Here we recall from the proof of Lemma~\ref{lem:rec} 
that if $\psi \in \QLS_{w}(\lambda)$, 
then $e_{i}^{\max}\psi \in \QLS_{r_{i}w}(\lambda)$. 
Hence it follows that 
$e_{i}^{\max}(e_{i}^{\max}\eta)=e_{i}^{\max}\eta$ 
is contained in $\QLS_{r_{i}w}(\lambda)$.
Therefore, we see from Lemma~\ref{lem:rec} that 
$f_{i}^{p}e_{i}^{\max}\eta \in \QLS_{w}(\lambda)$
for all $p \ge 0$ unless
$f_{i}^{p}e_{i}^{\max}\eta = \bzero$. 
From this, we conclude that $S_{i}(\eta)  
\subset \QLS_{w}(\lambda)$, as desired. 

From \eqref{eq:string}, we deduce that 
$\QLS_{w}(\lambda)$ decomposes into a disjoint union of $\alpha_{i}$-strings: 
\begin{equation*}
\QLS_{w}(\lambda) = S^{(1)} \sqcup S^{(2)} \sqcup \cdots \sqcup S^{(n)}, 
\quad 
\text{where $S^{(m)}$ is an $\alpha_{i}$-string for each $1 \le m \le n$}.
\end{equation*}
Since $i \in I$, the degree function $\Deg$ is 
constant on $S^{(m)}$ for each $1 \le m \le n$ (see \cite[(4.2)]{LNSSS2}); 
we set $d_{m}:=\Deg|_{S^{(m)}}$
for $1 \le m \le n$. Then we have 
\begin{equation*}
\gch \QLS_{w}(\lambda) = 
\sum_{m=1}^{n} q^{-d_{m}} \sum_{\eta \in S^{(m)}} e^{\wt(\eta)}.
\end{equation*}

Next, let us consider the intersection 
$\QLS_{r_{i}w}(\lambda) \cap S^{(m)}$ 
for each $1 \le m \le n$. Recall that 
if $\psi \in \QLS_{w}(\lambda)$, then 
$e_{i}^{\max}\psi \in \QLS_{r_{i}w}(\lambda)$. 
Since $S^{(m)} \subset \QLS_{w}(\lambda)$, 
it follows from the above that 
$\QLS_{r_{i}w}(\lambda)$ contains 
a unique element $\eta_{m} \in S^{(m)}$ 
such that $e_{i}\eta_{m} = \bzero$; 
in particular, $\QLS_{r_{i}w}(\lambda) \cap S^{(m)} \ne \emptyset$. 
Therefore, from Lemma~\ref{lem:string}, we deduce that
\begin{equation*}
\QLS_{r_{i}w}(\lambda) \cap S^{(m)} = 
 \bigl\{\eta_{m}\bigr\} \text{ or } S^{(m)}
\quad \text{for each $1 \le m \le n$}; 
\end{equation*}
here we assume that 
\begin{equation*}
\QLS_{r_{i}w}(\lambda) \cap S^{(m)} = 
\begin{cases}
\bigl\{\eta_{m}\bigr\} & \text{for $1 \le m \le p$}, \\[1.5mm]
S^{(m)} & \text{for $p+1 \le m \le n$}, 
\end{cases}
\end{equation*}
for some $0 \le p \le n$ for simplicity of notation. 
Then, we have 
\begin{equation*}
\gch \QLS_{r_{i}w}(\lambda) = 
\sum_{m=1}^{p} q^{-d_{m}} e^{\wt(\eta_{m})}+
\sum_{m=p+1}^{n} q^{-d_{m}} \sum_{\eta \in S^{(m)}} e^{\wt(\eta)}.
\end{equation*}
Combining all the above, we compute:
\begin{align*}
D_{i} \gch \QLS_{r_{i}w}(\lambda) & = 
\sum_{m=1}^{p} q^{-d_{m}} D_{i} e^{\wt(\eta_{m})} +
\sum_{m=p+1}^{n} q^{-d_{m}} D_{i} 
\Biggl(
  \underbrace{\sum_{\eta \in S^{(m)}} e^{\wt(\eta)}}_{%
  \begin{subarray}{c}
  =D_{i}e^{\wt(\eta_{m})} \\[1mm]
  \text{by \eqref{eq:Demazure}}
  \end{subarray}
  }
\Biggr) \\[1mm]
& = 
\sum_{m=1}^{p} q^{-d_{m}} D_{i} e^{\wt(\eta_{m})} +
\sum_{m=p+1}^{n} q^{-d_{m}} 
\underbrace{D_{i}D_{i}e^{\wt(\eta_{m})}}_{ =D_{i}e^{\wt(\eta_{m})} } \\[3mm]
& = \sum_{m=1}^{n} q^{-d_{m}} D_{i} e^{\wt(\eta_{m})} 
  = \sum_{m=1}^{n} q^{-d_{m}} \sum_{\eta \in S^{(m)}} e^{\wt(\eta)} 
  \quad \text{by \eqref{eq:Demazure}} \\[3mm]
& = \gch \QLS_{w}(\lambda). 
\end{align*}
This completes the proof of the proposition. 
\end{proof}
%
%
\subsection{Recursive formula for $\Mac{w\lambda}$.}
\label{subsec:demM-Mac0}

In view of Theorem~\ref{thm:Mac0}, 
Proposition~\ref{prop:dem} is equivalent to 
the following proposition.
%
%
\begin{prop} \label{prop:demM}
Let $w \in W^{J}$ and $i \in I$ be such that $w > r_{i}w$; 
note that $r_{i}w \in W^{J}$ by {\rm \cite[Lemma~5.8]{LNSSS1}}.
Then we have
\begin{equation*}
\Mac{w\lambda} = D_{i} \Mac{r_{i}w\lambda}.
\end{equation*}
\end{prop}

We can also show this proposition 
by using the polynomial representation of 
the double affine Hecke algebra as follows. 

\begin{proof}
Note that $r_{i}w \in W^{J}$ and $\pair{\alpha_{i}^{\vee}}{r_{i}w\lambda} > 0$ 
by \cite[Lemmas 5.8 and 5.9]{LNSSS1}. We set $\mu:=r_{i}w\lambda$. 
Since $\pair{\alpha_{i}^{\vee}}{\mu} = 
\pair{\alpha_{i}^{\vee}}{r_{i}w\lambda} > 0$ as seen above, 
it follows from \cite[(5.10.7)]{Mac} (in the notation thereof) that
\begin{equation} 
\left(t^{2}r_{i}+(t^{2}-1)\frac{1-r_{i}}{ 1-e^{\alpha_{i}} }-
   (t^{2}-1)\frac{1}{1-Y^{-\alpha_{i}}}\right) \cdot \Mact{\mu} 
= \Mact{r_{i}\mu}. \label{eq:demM1}
\end{equation}
Also, we know from \cite[(5.2.2')]{Mac} 
(in the notation thereof) that 
\begin{equation*}
Y^{-\alpha_{i}}\Mact{\mu} = 
q^{ \pair{\alpha_{i}^{\vee}}{\mu} } t^{ -2 \pair{v(\mu)\alpha_{i}^{\vee}}{\rho} }
\Mact{\mu}. 
\end{equation*}
Since $\pair{\alpha_{i}^{\vee}}{\mu} > 0$ as seen above, it follows that 
$\pair{v(\mu)\alpha_{i}^{\vee}}{v(\mu)\mu} = \pair{\alpha_{i}^{\vee}}{\mu} > 0$. 
Since $v(\mu)\mu$ is antidominant by the definition of $v(\mu)$, 
we see that $v(\mu)\alpha_{i}^{\vee}$ is a negative coroot, and hence 
$-2 \pair{v(\mu)\alpha_{i}^{\vee}}{\rho} > 0$. 
Therefore, by taking the limit $t \to 0$, we deduce from 
\eqref{eq:demM1} that 
\begin{equation*}
\left( \frac{r_{i}-1}{1-e^{\alpha_{i}}} + 1 \right) \Mac{\mu} = \Mac{r_{i}\mu}. 
\end{equation*}
We see by direct computation that 
\begin{equation*}
\frac{r_{i}-1}{1-e^{\alpha_{i}}} + 1 = D_{i}.
\end{equation*}
Thus, we obtain 
$D_{i} \Mac{\mu} = \Mac{r_{i}\mu}$, and hence 
$\Mac{w\lambda} = D_{i} \Mac{r_{i}w\lambda}$, as desired. 
This proves Proposition~\ref{prop:demM}.
\end{proof}
\begin{rem}
If we employ the special case $w=e$ of Theorem~\ref{thm:Mac0} as 
the start of the induction and use Proposition~\ref{prop:dem} and 
Proposition~\ref{prop:demM} (proved as above) in the induction step, 
then we can give an inductive proof of Theorem~\ref{thm:Mac0}. 
\end{rem}

%

\newpage

\end{document}